\documentclass[11pt,oneside, reqno]{amsart}
\usepackage[top=25mm, bottom=25mm, left=20mm, right=20mm]{geometry}

\usepackage{amsfonts,amssymb,amsmath,amsthm}
\usepackage[T1]{fontenc}
\usepackage[utf8]{inputenc}
\usepackage{mathtools}
\usepackage{graphics}
\usepackage{color}
\usepackage{enumitem}
\usepackage[unicode,hidelinks,bookmarks]{hyperref}
\usepackage{tikz}
\usepackage{float}
\usepackage{setspace}
\usetikzlibrary{arrows}

\usepackage{xcolor}

\usepackage{bm}

\usepackage{nameref}
\hypersetup{colorlinks = true, urlcolor = blue, linkcolor = blue, citecolor = red}

\allowdisplaybreaks

\numberwithin{equation}{section}
\newtheorem{theorem}[equation]{Theorem}

\newtheorem{lemma}[equation]{Lemma}
\newtheorem{proposition}[equation]{Proposition}
\newtheorem{conjecture}[equation]{Conjecture}
\theoremstyle{definition}

\newtheorem{example}[equation]{Example}

\newcommand{\BB}{\mathbb{B}}
\newcommand{\RR}{\mathbb{R}}
\newcommand{\NN}{\mathbb{N}}
\newcommand{\CC}{\mathbb{C}}
\newcommand{\DD}{\mathbb{D}}
\newcommand{\EE}{\mathbb{E}}
\newcommand{\II}{\mathbb{I}}
\newcommand{\JJ}{\mathbb{J}}

\newcommand{\TT}{\mathbb{T}}

\newcommand{\VV}{\mathbb{V}}
\newcommand{\ZZ}{\mathbb{Z}}

\newcommand{\dif}{\mathrm{d}}

\newcommand{\calA}{\mathcal{A}}
\newcommand{\calF}{\mathcal{F}}
\newcommand{\calH}{\mathcal{H}}
\newcommand{\calI}{\mathcal{I}}
\newcommand{\calL}{\mathcal{L}}
\newcommand{\calM}{\mathcal{M}}

\newcommand{\calS}{\mathcal{S}}

\newcommand{\TTT}{\mathcal{T}}
\newcommand{\PPP}{\mathcal{P}}

\DeclarePairedDelimiter\abs{\lvert}{\rvert}
\DeclarePairedDelimiter\norm{\lVert}{\rVert}

\DeclarePairedDelimiterX\Set[1]\{\}{#1}

\newcommand{\one}{\mathbf{1}}
\newcommand{\ind}[1]{\one_{#1}}

\begin{document}

\title[Polynomial ergodic theorems in the spirit of Dunford and Zygmund]{Polynomial ergodic theorems in the spirit of\\ Dunford and Zygmund}

\author{Dariusz Kosz}
\address{\scriptsize Dariusz Kosz (\textnormal{dariusz.kosz@pwr.edu.pl}) \newline
	Wroc{\l}aw University of Science and Technology, Wybrze{\.z}e Stanis{\l}awa Wyspia{\'n}skiego 27, 50-370 Wroc{\l}aw, Poland
    \vspace{-0.8em} 
}

\author{Bartosz Langowski}
\address{\scriptsize Bartosz Langowski (\textnormal{blangowski@franciscan.edu}) \newline
	Franciscan University of Steubenville, 1235 University Blvd., Steubenville, OH 43952, USA
    \vspace{-0.8em}
}

\author{Mariusz Mirek}
\address{\scriptsize Mariusz Mirek (\textnormal{mariusz.mirek@rutgers.edu}) \newline
	Rutgers University, Piscataway, NJ 08854-8019, USA
\newline Uniwersytet Wroc{\l}awski, Plac Grunwaldzki 2/4, 50-384 Wro\-c{\l}aw, Poland
\vspace{-0.8em}
}

\author{Pawe{\l} Plewa}
\address{\scriptsize Pawe{\l} Plewa (\textnormal{pawel.plewa@pwr.edu.pl}) \newline
	Wroc{\l}aw University of Science and Technology, Wybrze{\.z}e Stanis{\l}awa Wyspia{\'n}skiego 27, 50-370 Wroc{\l}aw, Poland
\newline Politecnico di
Torino, Corso Duca degli Abruzzi 24, 10129 Torino, Italy 
}

\begin{abstract}
The main goal of the paper is to prove convergence in norm and
pointwise almost everywhere on $L^p$, $p\in (1,\infty)$, for certain
multiparameter polynomial ergodic averages in the spirit of Dunford
and Zygmund for continuous flows. We will pay special attention to
quantitative aspects of pointwise convergence phenomena from the point
of view of uniform oscillation estimates for multiparameter polynomial Radon
averaging operators. In the proof of our main result we develop
flexible Fourier methods that exhibit and handle the so-called
\textit{``parameters-gluing''} phenomenon, an obstruction that arises in studying 
oscillation and variation inequalities for multiparameter polynomial
Radon operators. We will also discuss connections of our main result with 
a multiparameter variant of the Bellow--Furstenberg problem.

\end{abstract}

\thanks{Dariusz Kosz was supported by the Basque Government grant BERC 2022-2025, by the Spanish State Research Agency grant CEX2021-001142-S and by the National Science Centre of Poland grant SONATA BIS 2022/46/E/ST1/00036. Mariusz Mirek was partially supported by the NSF grant DMS-2154712 and by the NSF CAREER grant DMS-2236493. Pawe{\l} Plewa acknowledges the financial support of Compagnia di San Paolo.}

\dedicatory{\Large{In memory of Jacek Zienkiewicz our collaborator, friend and teacher.}}

\maketitle

\setcounter{tocdepth}{1}
\tableofcontents

\section{Introduction} \label{sec:1}

\subsection{Main result}

In 1951, Dunford \cite{D} and Zygmund \cite{Z} independently extended
Birkhoff’s pointwise ergodic theorem \cite{Bir} to the multiparameter
setting with multiple measure-preserving transformations. Their result
for measure-preserving flows can be formulated as the following theorem.

\begin{theorem}[Dunford \cite{D} and Zygmund \cite{Z}]
\label{thm:1}
Let $(X,\mathcal B(X), \mu)$ be a $\sigma$-finite measure space
equipped with a family of $k\in\ZZ_+$ measure-preserving flows
$(T_1^{t_1})_{t_1\in\RR},\ldots, (T_k^{t_k})_{t_k\in\RR} \colon X\to X$,
not necessarily commuting, and such that the mapping
$X \times \RR^k\ni (x,t_1,\ldots, t_k) \mapsto T_1^{t_1}\cdots T_k^{t_k}x\in X$
is measurable.  Then for every $p\in(1, \infty)$ and every
$f\in L^p(X)$ the limit for the Dunford--Zygmund averages
\begin{align}
\label{eq:10}
\lim_{\min\{M_1,\ldots, M_k\}\to\infty}\frac{1}{M_1 \cdots M_k}\int_0^{M_1}\cdots\int_0^{M_k}f\big(T_1^{t_1}\cdots T_k^{t_k}x\big)\dif t_k\ldots \dif t_1
\end{align}
exists $\mu$-almost everywhere on $X$ and in $L^p(X)$ norm.
\end{theorem}

For $k=1$ Theorem \ref{thm:1} coincides with the classical Birkhoff
pointwise ergodic theorem \cite{Bir} on $L^p(X)$ for
$p\in[1, \infty)$. However, when $k\ge2$, Theorem \ref{thm:1} is
evidently false at the endpoint $p=1$. This contrasts sharply with the
situation $k=1$ of Birkhoff's ergodic theorem, where pointwise
convergence in \eqref{eq:10} holds for general $f\in L^1(X)$. This
phenomenon is due to the fact that the averages in \eqref{eq:10} are
taken over rectangles $[0, M_1]\times \cdots \times [0, M_k]$ with
unrestricted side lengths $M_1,\ldots, M_k\in\RR_+$ and consequently
unrestricted convergence is allowed in Theorem \ref{thm:1} meaning
that $\min\{M_1,\ldots, M_k\}\to\infty$.  Such rectangles behave
differently than cubes $[0, M]^k$, i.e. when $M_1=\cdots=M_k=M$, for
which pointwise convergence in $L^1(X)$ holds for commuting flows by
Wiener's ergodic theorem \cite{W}. The situation here is similar to the
classical discussion on the failure at the endpoint $p=1$ of the Lebesgue
differentiation theorem for averages over rectangles with sides
parallel to the coordinate axes, see \cite{bigs} for a detailed discussion.

\medskip

Our main result, a polynomial extension of Theorem \ref{thm:1} (see
Theorem \ref{thm:main} below), is motivated by a recent progress that has
been made in \cite{BMSW} towards understanding a multiparameter
variant of the Bellow--Furstenberg problem (see Conjecture \ref{con:1}
below), as well as some investigations of uniform oscillation
estimates for multiparameter Radon averaging operators that arose in
\cite[Remarks below Theorem 6.37]{BMSW}.

\medskip

Before we formulate our main result we need to set up useful notation
and terminology.  Throughout the article the triple
$(X, \mathcal B(X), \mu)$ denotes a $\sigma$-finite measure space.
A measure-preserving transformation $T \colon X \to X$ is a map satisfying $\mu(T^{-1}[E]) = \mu(E)$ for each $E \in \mathcal B(X)$. A flow or $\RR$-flow $(T^t)_{t\in\RR} \colon X\to X$  is a group action of $\RR$ on $X$. In other words, a flow 
 is a collection of measure-preserving transformations $(T^t)_{t\in\RR}$ such that $T^sT^t=T^{s+t}$ for all $s, t\in\RR$ assuming that $T^0$ is the identity map  ${\rm id}$ of $X$. In order to avoid problems with measurability we will always assume that the mapping $X \times \RR\ni (x,t) \mapsto T^{t}x\in X$
is measurable. If we work with $d\in\ZZ_+$ measure-preserving $\RR$-flows
$(T_1^{t_1})_{t_1\in\RR},\ldots, (T_d^{t_d})_{t_d\in\RR} \colon X\to X$, we will always require 
that the mapping
\begin{align}
\label{eq:11}
X \times \RR^d\ni (x,t_1,\ldots, t_d) \mapsto T_1^{t_1}\cdots T_d^{t_d}x\in X
\end{align}
is measurable. In some cases, like in Theorem \ref{thm:1}, instead of group action, we will be allowed to consider semigroup actions of $\RR_+$ on $X$ and the conclusion remains unchanged. We will see more examples below.

The space of all formal $k$-variate polynomials
$P({\rm t}_1, \ldots, {\rm t}_k)$ with $k\in\ZZ_+$ indeterminates
${\rm t}_1, \ldots, {\rm t}_k$ and coefficients in a ring $\mathbb K$
will be denoted by $\mathbb K[{\rm t}_1, \ldots, {\rm t}_k]$. Each
polynomial $P\in\mathbb K[{\rm t}_1, \ldots, {\rm t}_k]$ will be also
identified with a function
$(t_1,\ldots, t_k)\mapsto P(t_1,\ldots, t_k)$ from $\mathbb K^k$ to
$\mathbb K$.  If ${\mathcal P} = \{P_1,\ldots, P_d\} \subset \mathbb K[\mathrm t_1, \ldots, \rm t_k]$, then $\deg \mathcal P=\max\{\deg P_j : j\in[d]\}$, where $\deg P_j$ is the degree of the polynomial $P_j\in \mathbb K[{\rm t}_1, \ldots, {\rm t}_k]$, and  $[N] \coloneqq  (0, N]\cap\ZZ$ for all $N\in\RR_+$.
We will mainly use $\mathbb K=\RR$ or $\mathbb K=\ZZ$ in this paper. 

\medskip

Let $d, k \in\ZZ_+$, and given a collection
${\mathcal T} = \{(T_1^{t_1})_{t_1\in\RR},\ldots, (T_d^{t_d})_{t_d\in\RR}\}$  of commuting
$\RR$-flows on $X$ satisfying \eqref{eq:11}, a locally integrable function $f$
on $X$, polynomials
${\mathcal P} = \{P_1,\ldots, P_d\} \subset \RR[\mathrm t_1, \ldots, \rm t_k]$, and a
vector $M = (M_1,\ldots, M_k)\in\RR_+^k$, we define a continuous  multiparameter polynomial ergodic
average by
\begin{equation} \label{eq:operator}
\mathcal A_{{M}; X, {\mathcal T}}^{\mathcal P}f(x)\coloneqq  
\frac{1}{M_1 \cdots M_k} \int_{\prod_{j\in[k]}[0, M_j]} f(T_1^{P_1(t)} \cdots T_d^{P_d(t)}x) \, \dif t, \qquad x\in X.
\end{equation}
 We will often abbreviate $\calA_{M; X, {\mathcal T}}^{{\mathcal P}}$ to  $\calA_{M; X}^{{\mathcal P}}$ when the collection $\mathcal T$ is understood. Sometimes, 
depending on how explicit we want to be, we will write out the averages
\[
\calA_{M;X, \mathcal T}^{\mathcal P}f=\calA_{M_1,\ldots, M_k;X, \mathcal T}^{P_1,\ldots, P_d} f .
\]
Using this notation we see that the average from \eqref{eq:10} coincides with 
$\calA_{M_1,\ldots, M_k;X, \mathcal T}^{{\rm t}_1, \ldots, {\rm t}_k}$; we simply take $P_1(t_1,\ldots, t_k)=t_1,\ldots, P_k(t_1,\ldots, t_k)=t_k$ in  definition \eqref{eq:operator}.

By \eqref{eq:11} we can deduce two things. If 
$f \in L^p(X)$ with $p \in [1,\infty]$, then for $\mu$-almost every $x \in X$ the function  $\RR^k_+\ni M \mapsto \mathcal A_{{M}; X, {\mathcal T}}^{\mathcal P}f(x)$ is well defined and continuous on $\RR^k_+$, while for every $M \in \RR_+^k$ the function $X\ni x \mapsto \mathcal A_{{M}; X, {\mathcal T}}^{\mathcal P}f(x)$ is measurable and belongs to $L^p(X)$.

The main purpose of the paper is to prove an extension of the Dunford--Zygmund theorem for averages \eqref{eq:operator}
in place of \eqref{eq:10} in Theorem \ref{thm:1} by establishing multiparameter uniform oscillation estimates, see Section \ref{sec:notation} for the definition of oscillation seminorms. Our main ergodic theorem reads as follows.

\begin{theorem}[Quantitative multiparameter ergodic theorem for flows]
\label{thm:main}
Let $(X, \mathcal B(X), \mu)$ be a $\sigma$-finite \linebreak measure space with a family 
${\mathcal T} = \{(T_1^{t_1})_{t_1\in\RR},\ldots, (T_d^{t_d})_{t_d\in\RR}\}$ of commuting
$\RR$-flows on $X$ satisfying \eqref{eq:11}. Let $d, k\in\ZZ_+$ and polynomials ${\mathcal P} = \{P_1,\ldots, P_d\} \subset \RR[\mathrm t_1, \ldots, \rm t_k]$ be given. Let $f\in L^p(X)$ for some $p\in [1, \infty]$, and let $\calA_{M_1,\ldots, M_k; X, {\mathcal T}}^{{\mathcal P}}f$ be the multiparameter ergodic  average defined  in \eqref{eq:operator}.
\begin{itemize}
\item[\rm(i)] \textit{(Mean ergodic theorem)} If $p\in(1, \infty)$, then the averages
$\calA_{M_1,\ldots, M_k; X, {\mathcal T}}^{{\mathcal P}}f$ converge in $L^p(X)$ norm as $\min\{M_1,\ldots, M_k\}\to\infty$.

\item[\rm(ii)] \textit{(Pointwise ergodic theorem)} If $p\in(1, \infty)$, then the averages
$\calA_{M_1,\ldots, M_k; X, {\mathcal T}}^{{\mathcal P}}f$ converge pointwise almost everywhere on $X$ as $\min\{M_1,\ldots, M_k\}\to\infty$.

\item[\rm(iii)] \textit{(Maximal ergodic theorem)}
If $p\in (1,\infty]$, then one has
\begin{align}
\label{eq:1}
\big\|\sup_{M_1,\ldots, M_k\in\RR_+}|\calA_{M_1,\ldots, M_k; X, {\mathcal T}}^{{\mathcal P}}f|\big\|_{L^p(X)}\lesssim_{p, \deg\mathcal P}\|f\|_{L^p(X)}.
\end{align}

\item[\rm(iv)] \textit{(Uniform oscillation ergodic theorem)}
If $p\in(1, \infty)$, then one has
\begin{align}
\label{eq:2}
\qquad \qquad\sup_{J\in\ZZ_+}\sup_{I\in\mathfrak S_J(\RR_+^k) }\big\|O_{I, J}^2(\calA_{M_1,\ldots, M_k; X, {\mathcal T}}^{{\mathcal P}}f: M_1, \ldots, M_k\in\RR_+)\big\|_{L^p(X)}\lesssim_{p, \deg\mathcal P}\|f\|_{L^p(X)},
\end{align}
 see Section \ref{sec:notation} for a definition of the oscillation seminorm $O_{I, J}^2$. The implicit constant in \eqref{eq:1} and \eqref{eq:2} may depend on $p$ and $\deg\mathcal P$, but not on the coefficients of the polynomials from $\mathcal P$.
\end{itemize}
\end{theorem}

Some remarks about Theorem \ref{thm:main} are in order.

\begin{enumerate}[label*={\arabic*}.]

\item As it was shown in \cite[Proposition 2.8]{MSzW} oscillation estimate \eqref{eq:2} from item (iv) implies pointwise convergence from (ii). We also know  that oscillation inequality \eqref{eq:2} implies the maximal one \eqref{eq:1}, since by \cite[Proposition 2.6]{MSzW} we have the following estimate
\begin{gather}
\nonumber
\big\|\sup_{M_1,\ldots, M_k\in\RR_+}|\calA_{M_1,\ldots, M_k; X, {\mathcal T}}^{{\mathcal P}}f|\big\|_{L^p(X)}
\le \sup_{M_1,\ldots, M_k\in\RR_+}\big\|\calA_{M_1,\ldots, M_k; X, {\mathcal T}}^{{\mathcal P}}f\big\|_{L^p(X)}\\ 
+
\sup_{J\in\ZZ_+}\sup_{I\in\mathfrak S_J(\RR_+^k) }\big\|O_{I, J}^2(\calA_{M_1,\ldots, M_k; X, {\mathcal T}}^{{\mathcal P}}f: M_1, \ldots, M_k\in\RR_+)\big\|_{L^p(X)}.
\label{eq:domin}
\end{gather}
Thus the conclusion from item (iv) implies the conclusion from item
(iii). Inequality \eqref{eq:1} from item (iii) was established by
Ricci and Stein in \cite{RS} (see also \cite{duo} for some special
cases) for the shift operators, which in view of the Calder{\'o}n
transference principle \cite{Ca} is equivalent to inequality
\eqref{eq:1} for arbitrary commuting $\RR$-flows.  Although the
arguments in this paper are inspired by those in \cite{duo, RS}, we
will present a new, flexible approach to overcome several new
difficulties that arise in the more general setting of estimating
uniform oscillations, which were not at all apparent in the maximal
function or singular integral estimates \cite{duo, RS}. Finally, the
conclusions from items (ii) and (iii) combined with the dominated
convergence theorem yield the norm convergence from item (i). This
reduces the matter to proving oscillation inequality \eqref{eq:2}.

\item If each polynomial takes only positive values $P_j \colon \RR^k\to \RR_+$, then  in  Theorem \ref{thm:main}, instead of group action, we are allowed to consider semigroup actions of $\RR_+$ on $X$ and the conclusion remains unchanged. In other words, we are allowed to consider in  Theorem \ref{thm:main} a family 
${\mathcal T} = \{(T_1^{t_1})_{t_1\in\RR_+},\ldots, (T_d^{t_d})_{t_d\in\RR_+}\}$  of commuting
$\RR_+$-flows on $X$ satisfying \eqref{eq:11}.

\item Theorem \ref{thm:main} extends Theorem \ref{thm:1} to the
polynomial orbits in the class of commuting flows $\mathcal T$.  If
$d=k\in\ZZ_+$ and $P_1(t_1,\ldots, t_k)=t_1,\ldots, P_k(t_1,\ldots, t_k)=t_k$, then
$\calA_{M_1,\ldots, M_k;X, \mathcal T}^{P_1,\ldots, P_d}=\calA_{M_1,\ldots, M_k;X, \mathcal T}^{{\rm t}_1, \ldots, {\rm t}_k}$
coincides with the average in \eqref{eq:10}. In fact, Dunford \cite{D}
and Zygmund \cite{Z} theorems have very simple proofs based on a
two-step procedure. In the first step  maximal inequality \eqref{eq:1} is proved with 
$\calA_{M_1,\ldots, M_k;X, \mathcal T}^{{\rm t}_1, \ldots, {\rm t}_k}$ in place of $\calA_{M_1,\ldots, M_k; X, {\mathcal T}}^{{\mathcal P}}$ by noting that
\begin{align*}
\calA_{M_1,\ldots, M_k;X, \mathcal T}^{{\rm t}_1, \ldots, {\rm t}_k}
=
\calA_{M_k;X, \mathcal T_k}^{{\mathrm t}_k}\circ\cdots\circ\calA_{M_1;X, \mathcal T_1}^{{\mathrm t}_1}, \qquad
{\mathcal T}_j = \{(T_j^{t_j})_{t_j\in\RR}\}, \quad  j\in[k],
\end{align*}
and using iteratively $L^p(X)$ bounds (with $p\in(1, \infty]$) for
$\sup_{M_j\in\RR_+}|\calA_{M_j;X, T_j}^{\mathrm t_j}f(x)|$ for
$j\in[d]$. The order of composition in this identity  is important since the flows ${\mathcal T} = \{(T_1^{t_1})_{t_1\in\RR},\ldots, (T_d^{t_d})_{t_d\in\RR}\}$  do not need to
commute.    The second step is based on a suitable adaptation of the
telescoping argument (relying on the Riesz decomposition \cite{Riesz})  to the multiparameter setting and an application of the classical Birkhoff ergodic theorem to verify the pointwise convergence on some dense class of functions, see for instance \cite{Nevo, MSzW}   for more
details. 

\item The two-step procedure may also be used to establish pointwise convergence when $\deg \mathcal P\ge2$, but then there is no obvious candidate for a dense class to verify convergence in the second step. Here a remedy to overcome this obstacle is the multiparameter oscillation seminorm that not only implies, but also quantifies, pointwise convergence for averages \eqref{eq:operator}, see \cite[Proposition 2.8]{MSzW} for more details. This is a very efficient strategy to handle such problems in ergodic context.

\item Multiparameter oscillations \eqref{eq:2} were considered for
the first time in \cite{JRW} in the context of the Dunford--Zygmund
averages \eqref{eq:10} for commuting measure-preserving
transformations. In \cite{JRW} as well as in the proof of Theorem
\ref{thm:main} it is crucial that the flows from
${\mathcal T} = \{(T_1^{t_1})_{t_1\in\RR},\ldots, (T_d^{t_d})_{t_d\in\RR}\}$
commute. It is interesting to know if \eqref{eq:2} holds for
$\calA_{M_1,\ldots, M_k;X, \mathcal T}^{{\rm t}_1, \ldots, {\rm t}_k}$
with noncommutative flows $\mathcal T$.

\item Multiparameter oscillations \eqref{eq:2} for $d=k\ge2$ with $\calA_{M_1,\ldots, M_k;X, \mathcal T}^{P_1({\mathrm t}_1), \ldots, P_k({\rm t}_k)}$ and
arbitrary one-variable polynomials $\{P_1, \ldots, P_k\}\subset\RR[\rm t]$
were recently established in \cite{MSzW}. Theorem
\ref{thm:main}(i)--(iii) is attributed to Bourgain, though he has
never published this result. The third author learned about this
result from Bourgain in October 2016, which initiated some
investigations \cite{BMSW} towards understanding a multiparameter
variant of the Bellow \cite{Bel} and Furstenberg \cite{Fur}
problem. More about this problem will be discussed below, see Conjecture \ref{con:1}.
A key observation in Bourgain's argument, like in the proofs of Dunford and Zygmund theorems, 
is the identity
\begin{align}
\label{eq:8}
\calA_{M_1,\ldots, M_k;X, \mathcal T}^{P_1({\mathrm t}_1), \ldots, P_k({\rm t}_k)}= \calA_{M_1;X, \mathcal T_1}^{P_1({\rm t}_1)}\circ\cdots\circ\calA_{M_k;X, \mathcal T_k}^{P_k({\rm t}_k)}, \qquad
{\mathcal T}_j = \{(T_j^{t_j})_{t_j\in\RR}\}, \quad  j\in[k],
\end{align}
which was also critical in establishing \eqref{eq:2} in \cite{MSzW}
for $\calA_{M_1,\ldots, M_k;X, \mathcal T}^{P_1({\mathrm t}_1), \ldots, P_k({\rm t}_k)}$ and commutative
flows
${\mathcal T} = \{(T_1^{t_1})_{t_1\in\RR},\ldots, (T_k^{t_k})_{t_k\in\RR}\}$.
Finally, let us mention that in \cite{MSzW} only discrete averaging
operators have been discussed (i.e. integrals in \eqref{eq:operator}
are replaced by sums) but the arguments from \cite{MSzW} are
adaptable line by line in our context of flows along polynomials.

\item A simple modification of counterexamples from the Bergelson--Leibman paper \cite{BL} readily shows that the commutation assumption cannot be dropped in Theorem \ref{thm:main} if $\deg\mathcal P\ge2$. It may even fail in the
one-parameter situation for the averages of the form $\calA_{M_1,\ldots, M_k;X, \mathcal T}^{P_1,\ldots, P_d}$ when $M_1=\cdots=M_k$.
\item In general, if $d, k\ge2$, the ranges of parameters for which we have $L^p(X)$ bounds in Theorem \ref{thm:main} are sharp. There is an interesting exception, if $d=1$ and $k=2$ it was shown by Patel \cite{Pat1} that the maximal function corresponding to the  averaging operator from \eqref{eq:operator} is of weak type $(1, 1)$ and consequently these averages converge pointwise almost everywhere on $L^1(X)$. It is expected that the same should be true if $d=1$ and $k\ge2$. It is also interesting to know whether we have weak type $(1,1)$ bounds in \eqref{eq:2} if $d=1$ and $k\ge2$.
In the one-parameter case,  if $d=k=1$, by \cite{CRW} it is  known that the maximal function corresponding to the  averaging operator from \eqref{eq:operator} is of weak type $(1, 1)$, though the one-parameter $M_1=\cdots=M_k$ and  multidimensional case $d\ge2$ and $k\in\ZZ_+$ is widely open, see \cite{STW} and the references given there.

\end{enumerate}

As we have seen above, the multiparameter theory for averages \eqref{eq:operator}
with orbits along polynomials with separated variables as in
\eqref{eq:8} is well understood and can be easily derived from the
one-parameter theory \cite{MSzW}, but the situation of
genuinely $k$-variate polynomials is much more challenging. In general setting, it seems
that there is no simple procedure (like changing variables or interpreting
the average from \eqref{eq:operator} as a composition of simpler
one-parameter averages as in \eqref{eq:8}) that would lead us  to the situation where pointwise convergence is known.
Therefore, in this paper we continue  systematic investigations of multiparameter
oscillation seminorms \cite{BMSW, JRW, MSzW}, which are especially interesting 
in the context of operators \eqref{eq:operator} with genuinely $k$-variate polynomials
${\mathcal P} = \{P_1,\ldots, P_d\} \subset \RR[\mathrm t_1, \ldots, \rm t_k]$.

Currently, it seems that the oscillation
seminorm (this is in marked contrast to the one-parameter theory) is the only available tool that allows us to handle
efficiently multiparameter pointwise convergence problems. See Section \ref{sec:notation} for the definitions of oscillation and $\rho$-variation seminorms. Although in this paper we made an effort to build a basic theory of multiparameter $2$-variations, see Theorem \ref{thm:2} and Lemma \ref{lem:splitting}, 
to the best of our knowledge, it is still open how to define
appropriate multiparameter jumps or $\rho$-variations, which would
incorporate all important properties of one-parameter jump or variation seminorms, see \cite{JKRW, jsw, Lep, MSZ1, MSzW}, even at the martingales level, 
and would be useful
in multiparameter pointwise convergence problems.

\subsection{Continuous averages  vs discrete averages}
Theorem \ref{thm:main} arose from recent attempts undertaken by the third author with Bourgain, Stein and Wright  
\cite{BMSW} towards understanding a multiparameter variant of Bellow
\cite{Bel} and Furstenberg \cite{Fur} problem, which can be subsumed
under the following conjecture.

\begin{conjecture}
\label{con:1}
Let $d, k\in\ZZ_+$ be given and let $(X, \mathcal B(X), \mu)$ be a probability measure space endowed with a family  of invertible commuting  measure-preserving transformations $T_1,\ldots, T_d \colon X\to X$. Assume that $P_1,\ldots, P_d \in\ZZ[{\rm m}_1, \ldots, {\rm m}_k]$. Then for any $f\in L^{\infty}(X)$ the multiparameter polynomial ergodic averages
\begin{align}
\label{eq:13}
A_{M_1,\ldots, M_k; X, T_1,\ldots, T_d}^{P_1, \ldots, P_d}f(x)\coloneqq \frac{1}{M_1 \cdots M_k}\sum_{m\in \prod_{j\in[k]}[M_j]}f(T_1^{P_1(m)}\cdots T_d^{P_d(m)}x)
\end{align}
converge for $\mu$-almost every $x\in X$, as  $\min\{M_1,\ldots, M_k\}\to\infty$.
\end{conjecture}

Conjecture \ref{con:1} arose in the late 1980's when Bourgain completed his
groundbreaking papers \cite{B1, B2, B3} establishing polynomial pointwise
ergodic theorem. This conjecture was also motivated by seeking a common
generalization of the results of Dunford \cite{D} and Zygmund \cite{Z}
on the one hand (which generalize Birkhoff's original theorem
\cite{Bir}) and the equidistribution theorem of Arkhipov, Chubarikov and Karatsuba \cite{ACK} on the
other hand (which generalize Weyl's theorem \cite{Weyl} to multiple polynomials of several
variables).

Some remarks about Conjecture \ref{con:1} are in order.

\begin{enumerate}[label*={\arabic*}.]
\item If $d=k=1$ and $P_{1}(m)=m$, then Conjecture \ref{con:1} follows
from Birkhoff's ergodic theorem \cite{Bir}, which establishes the
almost everywhere convergence (as well as the norm convergence
\cite{vN}) for all functions $f\in L^p(X)$, with $p\in [1,\infty)$,
defined on any $\sigma$-finite measure space
$(X, \mathcal B(X), \mu)$.

\item If $d=k=1$ and $P_{1}\in\ZZ[\rm m]$ is arbitrary, then Conjecture
\ref{con:1} was solved by Bourgain \cite{B1, B2, B3} in the mid 1980's
giving an affirmative answer to the famous open problem of Bellow
\cite{Bel} and Furstenberg \cite{F}. In fact, Bourgain showed that the
almost everywhere convergence (as well as the norm convergence, see
also \cite{F}) holds for all functions $f\in L^p(X)$, with
$p\in(1, \infty)$, defined on any $\sigma$-finite measure space
$(X, \mathcal B(X), \mu)$. In contrast to the Birkhoff theorem, if
$P_{1}\in\ZZ[\rm m]$ and $\deg P_1\ge2$, then the pointwise
convergence at the endpoint for $p=1$ may fail as was shown by
Buczolich and Mauldin \cite{BM} for $P_{1}(m)=m^2$ and by LaVictoire
\cite{LaV1} for $P_{1}(m)=m^k$ for any $k\ge2$. This also contrasts
sharply with the continuous situation $d=k=1$ with arbitrary $P\in\RR[\rm t]$  for flows ${\mathcal T} = \{(T^{t})_{t\in\RR}\}$  where the corresponding maximal function
$\sup_{M\in\RR_+}|\calA_{M; X, {\mathcal T}}^{P({\rm t})}f|$ 
is of weak type $(1, 1)$ by \cite{CRW}. The simplest reason explaining this difference is that
the real line $\RR$ is dilation invariant and  many of the change of variables techniques which are
available for integrals are not available for sums. This can be easily seen for $P({t})={t}^2$, if $f\ge0$ and $f\in L^1(X)$ then by a simple change of variables one has
\begin{align*}
\sup_{M\in\RR_+}|\calA_{M; X, {\mathcal T}}^{{\rm t}^2}f|\lesssim \sup_{M\in\RR_+}|\calA_{M; X, {\mathcal T}}^{{\rm t}}f|
\end{align*}
and the weak type $(1, 1)$ boundedness follows \cite{bigs}. In the
discrete setting \eqref{eq:13} such a simple argument is not
available; in fact, not even true by \cite{BM, LaV1}.

\item If $d,k\in\ZZ_+$ and $P_{1},\ldots, P_{d}\in\ZZ[{\rm m}_1, \ldots, {\rm m}_k]$ are arbitrary polynomials and $M_1=\cdots=M_k$, i.e. one-parameter multidimensional situation, then
Conjecture \ref{con:1} follows for instance from \cite{MSZ3, MSS}. These papers also provide the sharpest possible quantitative information (jump, variation and uniform oscillation estimates) about the pointwise convergence phenomena giving, in particular,  
the almost everywhere convergence (as well as the norm convergence) for all functions $f\in L^p(X)$, with $p\in (1, \infty)$, defined on any $\sigma$-finite measure space $(X, \mathcal B(X), \mu)$.

\item As we have seen above, a genuinely multiparameter case $d=k\ge2$ of Conjecture
\ref{con:1} with linear orbits,
i.e. $P_{j}(m_1, \ldots, m_d)=m_j$ for $j\in[d]$, was established
independently by Dunford \cite{D} (for discrete averages \eqref{eq:13})  and Zygmund \cite{Z} (for continuous averages \eqref{eq:operator}) in the early
1950's. It also follows from \cite{D, Z} that the
almost everywhere convergence (as well as the norm convergence) holds for all functions $f\in L^p(X)$, with
$p\in(1, \infty)$, defined on any $\sigma$-finite measure space
$(X, \mathcal B(X), \mu)$ equipped with a not necessarily commuting family of measure-preserving
transformations $T_1,\ldots, T_d \colon X\to X$. One also knows that pointwise convergence fails if $p=1$.

\item Conjecture \ref{con:1} in the case $d=k\ge2$ with
$P_{1}(m_1, \ldots, m_k)\in \ZZ[{\rm m}_1],\ldots, P_{k}(m_1, \ldots, m_k)\in \ZZ[{\rm m}_k]$
for commuting measure-preserving transformations
$T_1,\ldots, T_d \colon X\to X$ is attributed to Bourgain.  Although
Bourgain's result has never been published, a simple argument
recovering Bourgain's result (in fact in the spirit of Bourgain's
original proof) was recently published in \cite{MSzW}. 

\item We emphasize that Conjecture \ref{con:1} has been recently 
solved by  the third author with Bourgain, Stein and Wright  \cite{BMSW} for $d=1$, $k\in\ZZ_+$ and for arbitrary $k$-variate polynomials $P_1\in\ZZ[{\rm m}_1, \ldots, {\rm m}_k]$. The methods of proof from \cite{BMSW} develop many new sophisticated methods  including  estimates for multiparameter exponential sums, and introduce new
ideas from the so-called multiparameter circle method in the context of the geometry of backwards
Newton diagrams that are dictated by the shape of the polynomials defining  ergodic averages from \eqref{eq:13}.
\end{enumerate}

Interestingly, Dunford theorem \cite{D} (for discrete averages \eqref{eq:13})  and Zygmund theorem \cite{Z} (for continuous averages \eqref{eq:operator})  are equivalent in the sense that one result can be deduced from the other and vice versa. This can be seen at the level of comparing corresponding maximal functions for suitably chosen measure-preserving systems and flows. If $\deg\mathcal P\ge2$, then there is no easy way to compare averages \eqref{eq:operator} with averages \eqref{eq:13}. This obstruction to a large extent motivates this paper on the one hand.

On the other hand, Conjecture \eqref{con:1} is widely open and even
though it was recently established for $d=1$, $k\in\ZZ_+$ in
\cite{BMSW} it is not clear if it is true in general case
$d, k\in\ZZ_+$ for all polynomials
$P_1,\ldots, P_d \in\ZZ[{\rm m}_1, \ldots, {\rm m}_k]$. In \cite{BMSW}
we encountered some obstructions of arithmetic nature in estimating
exponential sums corresponding to \eqref{eq:13}, which prevented us
from attacking Conjecture \ref{con:1} for all $d, k\in\ZZ_+$ and
general polynomials. This knowledge was a starting point to understand
whether the obstacles arising in \cite{BMSW} can be overcame for the
integral averages \eqref{eq:operator} associated with flows.  Beside
the results from \cite{BMSW}, our main result, Theorem \ref{thm:main},
gives some evidence (being an important ingredient in the proof) that
Conjecture \ref{con:1} may be true. We plan to pursue investigation
in this direction in the near future.
\subsection{General reductions and strategy of the proof}
We begin with certain reductions which will simplify the proof of
Theorem \ref{thm:main} or more precisely uniform oscillation estimates
\eqref{eq:2}. We claim that it suffices to prove Theorem
\ref{thm:main} for any finite family
$\mathcal P\subset\RR[\mathrm t_1, \ldots, \rm t_k]$, which consists
of monomials all having coefficient one.  Of course, if Theorem
\ref{thm:main} holds for arbitrary finite family
$\mathcal P\subset\RR[\mathrm t_1, \ldots, \rm t_k]$ then it holds for
any finite family of monomials all having coefficient one. Our aim
will be to reverse this implication.

We fix
$\mathcal P=\{P_1,\ldots, P_d\}\subset\RR[\mathrm t_1, \ldots, \rm t_k]$ and
${\mathcal T} = \{(T_1^{t_1})_{t_1\in\RR},\ldots, (T_d^{t_d})_{t_d\in\RR}\}$. Note that for every 
$\mathcal P\subset\RR[\mathrm t_1, \ldots, \rm t_k]$ there exists a finite set of multiindices $\Gamma\subseteq (\{0\}\cup[d_0])^k$ for some $d_0\in\ZZ_+$ such that 
\[
P_{j}(t) = \sum_{\gamma\in\Gamma} a_{j,\gamma} t^{\gamma}, \qquad t\in\RR^k, \quad j\in[d],
\]
for some coefficients $a_{j,\gamma}\in\RR$. Here we used the notation $t^{\gamma}\coloneqq  t_1^{\gamma_1}\cdots t_k^{\gamma_k}$ for any $t=(t_1, \ldots, t_k)\in\RR^k$. Define $T_{\gamma}^s \coloneqq  \prod_{j\in[d]}T_{j}^{a_{j,\gamma}s}$ for any $\gamma\in\Gamma$ and $s\in\RR$. Now using commutativity of $\mathcal T$ we see that
\[
\prod_{j\in[d]}T_j^{P_j(t)}
=\prod_{j\in[d]}\prod_{\gamma\in\Gamma}T_j^{a_{j,\gamma} t^{\gamma}}
=\prod_{\gamma\in\Gamma}T_{\gamma}^{t^{\gamma}}.
\]
Setting $\mathcal P_{\Gamma}\coloneqq \{t^\gamma: \gamma\in\Gamma\}$ and
$\mathcal T_{\Gamma}\coloneqq \{(T_{\gamma}^{s})_{s\in\RR_+}: \gamma\in\Gamma\}$
one has that
$\calA_{M_1,\ldots, M_k; X, {\mathcal T}}^{{\mathcal P}}f=\calA_{M_1,\ldots, M_k; X, {\mathcal T_{\Gamma}}}^{{\mathcal P_{\Gamma}}}f$,
which shows that it suffices to prove Theorem \ref{thm:main} (or more
precisely inequality \eqref{eq:2}) for any finite family of monomials
all having coefficient one. In fact, we performed a lifting argument, which originates in \cite{bigs}.

Before we describe our second reduction, certain transference, let us consider the following example.

\begin{example}\label{ex:1}
Important examples of \eqref{eq:operator} from the point of view of
pointwise convergence problems and various questions in harmonic analysis are
multiparameter Radon averaging operators. For given $d, k\in\ZZ_+$, and a locally integrable function $f \colon \RR^d\to \CC$,  polynomials ${\mathcal P} = \{P_1,\ldots, P_d\} \subset \RR[\mathrm t_1, \ldots, \rm t_k]$, and a
vector $M = (M_1,\ldots, M_k)\in\RR_+^k$ the multiparameter Radon averaging operator is defined by setting
\begin{align}
\label{eq:12}
\mathcal M_{M}^{\mathcal P}f(x)\coloneqq  
\frac{1}{M_1 \cdots M_k} \int_{\prod_{j\in[k]}[0, M_j]} f(x_1-P_1(t), \ldots, x_d-P_d(t)) \, \dif t, \qquad x=(x_1,\ldots, x_d)\in \RR^d.
\end{align}
Considering the
$d$-dimensional Euclidean space $(\RR^d, \mathcal L(\RR^d), \mu_{\RR^d})$ with Lebesgue  $\sigma$-algebra and Lebesgue measure and 
a collection of shifts $(S_1^{t_1})_{t_1\in\RR},\ldots, (S_d^{t_d})_{t_d\in\RR} \colon \RR^d\to\RR^d$,
defined by
$S_j^{t_j}(x) \coloneqq x-t_je_j$ for every $x\in\RR^d$ and $t_j\in\RR$ (here $e_j$ is $j$-th basis
vector from the standard basis in $\RR^d$ for each $j\in[d]$), we see that 
$\calA_{M; \RR^d, {\mathcal T}}^{{\mathcal P}}$ with ${\mathcal T} = \{(T_1^{t_1})_{t_1\in\RR},\ldots, (T_d^{t_d})_{t_d\in\RR}\}=\{(S_1^{t_1})_{t_1\in\RR},\ldots, (S_d^{t_d})_{t_d\in\RR}\}$ coincides with 
$\mathcal M_{M}^{\mathcal P}$.

\end{example}

There is now an extensive literature concerning one-parameter
operators of Radon type ($k=1$ in \eqref{eq:12}) in the context of
maximal estimates \cite{CRW, CNSW, duo&rubio,bigs} as well as jump,
oscillation and variation inequalities \cite{JKRW, jsw, MSZ1,
bootstrap, MSzW} that (as it is well known \cite{B3, jsw}) quantify
pointwise convergence phenomena in $L^p$ spaces, see Section
\ref{sec:notation} for the definitions of oscillations and variations
and the recent survey \cite{MSzW} for a more detailed exposition of
the subject. Situation is dramatically different with multiparameter
operators of Radon type ($k\ge2$ in \eqref{eq:12}). Although
multiparameter maximal functions and singular integrals have been
studied for many years \cite{CWW, CWW1, duo, Pat1, Pat2, Pat3, RS,
bigs, Street}, the context of quantitative bounds (like jump,
oscillation and variation inequalities) in the multiparameter
setting (see also Section~\ref{sec:notation} for definitions) seems
as a completely new and highly unexplored territory \cite{BMSW, JRW,
MSzW}. This paper is the first treatment developing this theory for
multiparameter operators of Radon type.

Multiparameter Radon operators $\mathcal M_{M}^{\mathcal P}$
will play an essential role in the proof of Theorem \ref{thm:main}. Taking into account Example \ref{ex:1}, and using   the
Calder{\'o}n transference principle \cite{Ca} (see \cite[Proposition 5.1]{BMSW1} for a similar argument and more details), inequality \eqref{eq:2} can be reduced to the corresponding uniform oscillation inequality for Radon operators $\mathcal M_{M}^{\mathcal P}$. From now on we will focus on proving our second main result of this paper, which implies Theorem \ref{thm:main} and reads as follows. 

\begin{theorem}[Uniform oscillation estimates for multiparameter Radon operators]
\label{thm:3}
Suppose that
$\mathcal P=\{P_1,\ldots, P_d\}\subset\RR[\mathrm t_1, \ldots, \rm t_k]$
is a collection of monomials all having coefficient one.  Then for any
$p \in (1, \infty)$ and the family of multiparameter Radon operators
$\{\mathcal M^{\PPP}_M : M\in\RR_+^k\}$ defined in \eqref{eq:12} the
following uniform oscillation inequality holds
\begin{equation}
\label{eq:thm3}
	\sup_{J\in\ZZ_+}\sup_{I \in \mathfrak S_J(\RR^k_+)}
	\big\| O^2_{I, J} \big( \mathcal M^{\PPP}_M f : M\in\RR^k_+\big) \big\|_{L^p(\RR^d)} \lesssim_{p, \deg\PPP} \| f\|_{L^p(\RR^d)}, \qquad f\in L^p(\RR^d). 
	\end{equation}
The implicit constant in \eqref{eq:thm3} depends only on $p$ and the degree of $\PPP$.
\end{theorem}

We now briefly outline the key steps of the proof of Theorem \ref{thm:3}.
We fix a collection of monomials $\mathcal P=\{P_1,\ldots, P_d\}\subset\RR[\mathrm t_1, \ldots, \rm t_k]$
all having coefficient one. The proof is by induction on $d$.
We decompose 
\[
\calM^{\PPP}_M = {\rm main}(\calM^{\PPP}_M) + {\rm error}(\calM^{\PPP}_M),
\]
where ${\rm main}(\calM^{\PPP}_M)$ and ${\rm error}(\calM^{\PPP}_M)$ are convolution operators whose respective Fourier multipliers are $\sum_{D \subseteq [d] : D \neq \emptyset} (-1)^{|D|+1} \mathfrak m^{\PPP, D}_M$ and $\tilde{\mathfrak m}^\PPP_M$, see the decomposition
\eqref{eq:mudecomp}.
This decomposition has a combinatorial nature and  arises from a suitable application of the  inclusion--exclusion procedure for tensor products of lower-dimensional Radon operators and appropriately adjusted convolution operators acting on the remaining variables. For example, for $d=2$, it is imitated by the following symbolic identity
\[
\underbrace{{\rm R}_1 {\rm R}_2}_{\calM^{\PPP}_M} =  \underbrace{( {\rm C}_1 {\rm R}_2+ {\rm R}_1 {\rm C}_2 - {\rm C}_1 {\rm C}_2)}_{{\rm main}(\calM^{\PPP}_M)}+ \underbrace{({\rm R}_1 - {\rm C}_1)({\rm R}_2 - {\rm C}_2)}_{{\rm error}(\calM^{\PPP}_M)},  
\]
where the symbol ${\rm R}_i$ indicates that the lower-dimensional Radon operator occupies the $i$-th variable, while ${\rm C}_i$ represents the action of a convolution operator on the $i$-th variable. By induction we obtain
\begin{equation*}
	\sup_{J\in\ZZ_+}\sup_{I \in \mathfrak S_J(\RR^k_+)}
	\big\| O^2_{I, J} \big( {\rm main}(\calM^{\PPP}_M)f: M\in\RR^k_+\big) \big\|_{L^p(\RR^d)} \lesssim_{p, \deg\PPP} \| f\|_{L^p(\RR^d)}, \qquad f\in L^p(\RR^d). 
\end{equation*}      
The case $d=1$ follows by an argument in the spirit of \cite{MSzW}, where uniform oscillation estimates for bumps are critical. Although proving the inductive step requires some work, the most complicated part is to estimate the error term. Fortunately, the error term has some cancellations with respect to all axes, see \eqref{cond:2}, which with the aid of Theorem \ref{thm:2}  will imply the following inequality for $2$-variations
\begin{equation*} 
\big\| V^2 \big({\rm error}(\calM^{\PPP}_M)f : M\in\RR^k_+\big) \big\|_{L^p(\RR^d)} \lesssim_{p, \deg\PPP} \| f\|_{L^p(\RR^d)}, \qquad f\in L^p(\RR^d),
\end{equation*}
which is even larger object than the oscillation seminorm, see Section~\ref{sec:notation} for appropriate definitions.

Here we develop a basic theory of multiparameter variation estimates, see in particular Theorem \ref{thm:2} and Lemma \ref{lem:splitting}. To this end, we split the variation seminorm into long and short variations. A precursor of this splitting for $d = k = 1$ is a well-known numerical inequality
\begin{align}
\label{eq:14}
V^2(a_s : s \in \RR_+) \lesssim V^2(a_s : s \in \DD) + \Big( \sum_{n \in \ZZ} V^2 \big(a_s : s \in [2^n, 2^{n+1}) \big) ^2 \Big)^{\frac{1}{2}},
\end{align}
where $\DD \coloneqq  \{ 2^n : n \in \ZZ\}$ is the set of dyadic numbers. In  general,  the problem is much more subtle since
we may encounter the so-called \textit{``parameters-gluing''} phenomenon, a certain obstruction which prevents us from a naive 
application of the Littlewood--Paley theory to control long variations. To be more precise, recall that we have been working with $d$ monomials from $\mathcal P$, which determine certain ``directions''. A naive thought suggests that for each direction one should perform an appropriate Littlewood--Paley decomposition to control long variations by a square function. This idea is efficient if we have $d$ non-degenerate  directions, but fails if the number of directions is less than $d$. 
To overcome this difficulty  we introduce an auxiliary parameter $r \coloneqq  {\rm rank}(\mathbb{M}^{\PPP})$, where $\mathbb{M}^{\PPP}$ is a matrix of size $d \times k$ whose rows are the exponents of the monomials from $\PPP$.
There is a one-to-one correspondence between these exponents and the ``directions'', and the number of nondegenerate  directions is $r\le \min\{d, k\}$. Hence, $r$ measures how many exponents or directions have been glued together and explains the \textit{``parameters-gluing''} phenomenon.  

Once $r$ has been defined, we perform a multiparameter
counterpart of splitting from \eqref{eq:14} for appropriately defined
$r$-dimensional dyadic grid, see Lemma \ref{lem:splitting}.  It is
especially challenging to show that the resulting $r$-dimensional
objects can control the $k$-dimensional variation seminorm even if
$r < k$. This is the core of our argument. Then the long variations
are handled by the square function techniques, the
Littlewood--Paley theory adjusted to the $r$ non-degenerate
``directions'', and bootstrapping arguments from \cite{duo&rubio,
bootstrap}. On the other hand, a multiparameter
Rademacher--Menshov argument, see inequality \eqref{eq:num},
allows us to control the short variations on dyadic cubes
by a countable family of square functions. Finally, we use a delicate
bootstrap--interpolation technique developed in \cite{bootstrap} to
bound these square functions.

\subsection{Structure of the paper} \label{sec:1stru}
The rest of the paper is organized as follows. In Section~\ref{sec:notation} we introduce basic notation.
In Section~\ref{sec:var} we obtain variational estimates for abstract families of operators satisfying several natural assumptions. We decompose the variation seminorm into suitably defined long and short variations; the first capturing  behavior of a sequence when it moves between dyadic blocks and the latter measuring the variations inside each block.
All arguments are adjusted to dyadic grids determined by the underlying nonisotropic dilations.
In Section~\ref{sec:rad} we prove  Theorem~\ref{thm:3}, which is the key result of the paper.
The proof runs by induction over the dimension $d$. For $d=1$ the result follows readily from the general theory
of uniform oscillation estimates recently developed in \cite{MSzW}.

\subsection{In Memoriam} It was very sad  when we learned that Jacek Zienkiewicz (01.01.1967 -- 09.01.2023) passed away. Jacek was a prominent member of the group of harmonic analysis in the Mathematical Institute at the University of Wrocław. He had an important impact on our careers as well as careers of many generations of young mathematicians in Wrocław. Jacek was always open-minded, loved mathematics endlessly, and transmitted to students and collaborators his enthusiasm for mathematics. It is an irreplaceable loss to the whole harmonic analysis community. We miss our friend dearly.

\section{Notation}\label{sec:notation}
In this section we set up our notation that will be used throughout the paper. 

\subsection{Basic notation} The sets $\ZZ$, $\RR$,
$\CC$, and $\TT \coloneqq   \RR/\ZZ$ have standard meaning. We let $\ZZ_+ \coloneqq   \{1, 2, \ldots\}$, $\NN \coloneqq   \{0,1,2,\ldots\}$, $2\NN \coloneqq   \{0,2,4,\ldots\}$, and
$\RR_+ \coloneqq   (0, \infty)$. We also consider the set of dyadic numbers $\DD \coloneqq   \{2^n: n\in\ZZ\}$ and, for every $N\in\ZZ_+$, the block $[N] \coloneqq  (0,N] \cap \ZZ = \{1, \ldots, N \}$. 

By  $\ind{E}$ we denote the indicator function of a set $E$. If $E$ is finite, then $|E|$ is the number of its elements, while $\EE_{m \in E}f(m)$ denotes the average value of a function $f \colon E\to \CC$ over $E\neq\emptyset$ taken with respect to counting measure on $E$.
For a statement $S$ we write $\ind{S}$ to denote its indicator. For instance, $\ind{x\in A} = \ind{A}(x)$.

For two nonnegative quantities $A, B$ we write $A \lesssim B$ if there
is an absolute constant $C \in \RR_+$ such that $A\le CB$. If we want to emphasize that $C$ depends on a parameter $\alpha$, then we write $A \lesssim_{\alpha} B$. 

\subsection{Multiparameter notation} For $s=(s_1,\ldots, s_k), u=(u_1,\ldots, u_k) \in \RR^k_+$ and $\lambda \in \RR$ we use the following  notation $s \otimes u \coloneqq   (s_1 u_1, \dots, s_k u_k)$ and $\lambda \odot s \coloneqq   (\lambda s_1, \dots, \lambda s_k)$.  
In particular, we have
\[
\calA_{s}^{\PPP, \TTT} f(x) = \int_{u \in (0,1)^k} f(T_1^{P_1(s \otimes u)}\cdots T_d^{P_d(s \otimes u)}x) \, \dif u
\qquad {\rm and} \qquad
\mathcal M^{\PPP}_s f(x) =  \int_{(0,1)^k} f(x - P(s \otimes u)) \, \dif u.
\]
We will also write $u^{-1}\coloneqq (u_1^{-1},\ldots, u_k^{-1})$ for any $ u=(u_1,\ldots, u_k) \in \RR^k_+$.
We shall abbreviate the $k$-tuples $(0,\dots,0)$ and $(1,\dots,1)$ to $\bm 0$ and $\bm 1$ respectively. For $n \in \ZZ^k$ we write ${\bf 2}^n \coloneqq  (2^{n_1}, \dots, 2^{n_k})$.

\subsection{Difference and shift operators} Let $\{e_j\colon j\in[k]\}$ be the standard basis in $\RR^k$. For $\{a_s\colon s\in\RR^k_+\}\subseteq \CC$ and $h\in\RR^k_+$ we define respectively the partial  difference and shift operators by setting
\begin{align*}
\Delta^{\{j\}}_ha_s\coloneqq a_{s+e_j\otimes h}-a_s, \qquad \text{ and } \qquad
{\rm T}^{\{j\}}_ha_s\coloneqq a_{s+e_j\otimes h}, \qquad  j\in[k].
\end{align*}
For a set $K=\{j_1,\ldots, j_m\}\subseteq[k]$ we will write
\begin{align*}
\Delta^{K}_ha_s\coloneqq \Delta^{\{j_1\}}_h\circ\cdots\circ \Delta^{\{j_m\}}_h  a_s,
\qquad \text{ and } \qquad {\rm T}^{K}_ha_s\coloneqq {\rm T}^{\{j_1\}}_h\circ\cdots\circ {\rm T}^{\{j_m\}}_h  a_s.
\end{align*}
We also let
\begin{align*}
\Delta^{\emptyset}_ha_s={\rm T}^{\emptyset}_ha_s=a_s.
\end{align*}

\subsection{Euclidean spaces} For any $x=(x_1,\ldots, x_d)$,
$\xi=(\xi_1, \ldots, \xi_d)\in\RR^d$, we use the following symbols
\begin{align*}
x\cdot\xi \coloneqq   \sum_{i \in [d]} x_i \xi_i, 
\qquad |x|_1 \coloneqq  \sum_{i \in [d]} |x_i|,  
\qquad |x|_{\infty} \coloneqq   \max_{i\in[d]}|x_i|,
\end{align*}
to denote respectively the standard inner product, the taxicab norm,  
and the maximum norm on $\RR^d$. 

\subsection{Function spaces}
In the paper all vector spaces are defined over $\CC$.
For a continuous linear map $T \colon B_1 \to B_2$ between two normed
vector spaces $B_1$ and $B_2$, its operator norm is denoted by
$\|T\|_{B_1 \to B_2}$.

Throughout the paper a triple $(X, \mathcal B(X), \mu)$ will denote a measure space with a~set $X$, a~$\sigma$-algebra $\mathcal B(X)$, and a~$\sigma$-finite measure
$\mu$. The space of all $\mu$-measurable functions $f \colon X\to\CC$ is denoted by $L^0(X)$. Given $p\in[1, \infty)$ we denote by $L^p(X)$ the space consisting of all functions $f \in L^0(X)$ such that the $p$-th power of the modulus of $f$ is integrable. For $p = \infty$ we write $L^{\infty}(X)$ for the space of all
essentially bounded functions in $L^0(X)$. According to this, $L^{p}(X)$ is a Banach space with its norm given by 
\[
\norm{f}_{L^p(X)} \coloneqq   
\begin{cases} \left(\int_{X}|f(x)|^p \, \dif \mu(x) \right)^{\frac{1}{p}} & {\rm if} \ p \in [1, \infty),\\
{\rm ess} \sup_{x\in X} |f(x)| & {\rm if} \ p = \infty.
\end{cases}
\]
We extend this notion to functions taking values in a separable normed vector space $B$. For instance,
\begin{align*}
L^{p}(X;B)
 \coloneqq   \big\{F\in L^0(X;B):\|F\|_{L^{p}(X;B)} \coloneqq   \big\| \|F\|_B \big\|_{L^{p}(X)}<\infty\big\},
\end{align*}
where $L^0(X;B)$ is the space of all measurable functions from $X$ to
$B$ (up to the almost everywhere equivalence). 
We also work with the so-called weak-$L^p$ spaces defined similarly with the aid of the quasinorms
\[
\norm{f}_{L^{p,\infty}(X)} \coloneqq   
\begin{cases} \sup_{\lambda \in \RR_+}\lambda\mu(\{x\in X:|f(x)|>\lambda\})^{\frac{1}{p}} & {\rm if} \ p \in [1, \infty),\\
{\rm ess} \sup_{x\in X} |f(x)| & {\rm if} \ p = \infty.
\end{cases}
\]
In our case $X$ is usually $\RR^d$ equipped with Lebesgue measure or a countable set $X$ equipped with counting measure. We then write $L^p(\RR^d)$ or $\ell^p(X)$ instead of $L^p(X)$ and so on.

\subsection{Fourier transform}  
We use a convenient notation $\bm e(z)\coloneqq e^{2\pi {\bm i} z}$ for every $z\in\CC$, where ${\bm i}^2=-1$. Let $\calF_{\RR^d}$ denote the Fourier transform on $\RR^d$ defined for
any $f \in L^1(\RR^d)$ and for any $\xi\in\RR^d$ as
\begin{align*}
\calF_{\RR^d} f(\xi) \coloneqq   \int_{\RR^d} f(x) \bm e(x\cdot\xi) \, \dif x.
\end{align*}
More generally, for a finite measure $\nu$ on $\RR^d$ its Fourier transform is given by
\[
\calF_{\RR^d} \nu(\xi) \coloneqq   \int_{\RR^d} \bm e(x \cdot \xi) \, \dif \nu(x).
\]

The Fourier multiplier operator associated with a bounded measurable function $\mathfrak m \colon \RR^d \to\CC$ is given by 
\begin{align*} 
T_{\RR^d}[\mathfrak m]f(x) \coloneqq   \int_{\RR^d} \mathfrak m(\xi) \calF_{\RR^d}f(\xi) \bm e(-\xi\cdot x)  \, \dif\xi.
\end{align*} 
In particular, for $s \in \RR_+^k$ and a polynomial mapping $\PPP \colon \RR_+^k \to \RR^d_+$ corresponding Radon operator $\calM_s^\PPP$ can be expressed as a convolution with a probability measure $\nu_{s}^\PPP$ which satisfies
\begin{equation*}
\calF_{\RR^d} \nu_{s}^\PPP(\xi)=\int_{(0,1)^k} \bm e \big(P(s \otimes u)\cdot \xi\big) \, \dif u, \qquad \xi\in\RR^d.
\end{equation*}
We thus see that
$
\calM_s^\PPP = T_{\RR^d}[\mathfrak m_s^\PPP],
$
where $\mathfrak m_s^\PPP \coloneqq  \calF_{\RR^d} \nu_{s}^\PPP$.

\subsection{Coordinatewise order} For any $x=(x_1,\ldots, x_k)\in\RR^k$ and $y=(y_1,\ldots, y_k)\in\RR^k$ we say $x\preceq y$ if and only if $x_i\le y_i$ for each $i\in[k]$. Similarly, we write $x\prec y$ if and only if $x_i< y_i$ for each $i\in[k]$. 

If $\II \subseteq \RR^k_+$ is a set of indices, then we define
\begin{align*}
\mathfrak S (\II) \coloneqq  \Set[\big]{(I_j:j\in\NN) \subset \II : I_{0}\prec 
I_{1} \prec \cdots }
\end{align*}
which is a family of all strictly increasing sequences (with respect to the coordinatewise order) taking their values in $\II$. We also consider its truncation defined for each $J \in \ZZ_+$ by
\begin{align*}
\mathfrak S_J(\II) \coloneqq  \Set[\big]{(I_j:j\in \{0\} \cup [J])\subset \II : I_{0}\prec I_{1}\prec \cdots \prec I_{J} }.
\end{align*}

\subsection{Oscillation and variation seminorms} For $x \in X$ let $(a_{s}(x): s\in\RR^k)\subseteq\CC$ be a family of numbers. Given $\rho \in [1, \infty)$, $\II, \JJ\subset \RR^k$, and $I=(I_j : j\in\NN) \in \mathfrak S(\II)$ we define the $k$-parameter
$\rho$-oscillation seminorm
\begin{align*}
O_{I}^\rho(a_{s}(x): s \in \JJ) \coloneqq  
\Big(\sum_{j \in \ZZ_+}\sup_{s\in \BB_j[I] \cap \JJ}
\abs[\big]{a_{s}(x) - a_{I_{j-1}}(x)}^\rho\Big)^{\frac{1}{\rho}},
\end{align*}
where
$\BB_j[I] \coloneqq   [I_{(j-1),1}, I_{j,1})\times\dots\times[I_{(j-1),k}, I_{j,k})$
are boxes determined by the elements $I_j=(I_{j,1}, \ldots, I_{j,k})$ of
the sequence $I$. We also consider its truncation defined for each $J \in \ZZ_+$ and $I \in \mathfrak S_J (\II)$ by
\begin{align*}
O_{I,J}^\rho(a_{s}(x): s \in \JJ) \coloneqq  
\Big(\sum_{j \in [J]} \sup_{s\in \BB_j[I] \cap \JJ}
\abs[\big]{a_{s}(x) - a_{I_{j-1}}(x)}^\rho\Big)^{\frac{1}{\rho}}.
\end{align*}
Similarly, we introduce a $k$-parameter variant of $\rho$-variations by
\begin{align*}
V^{\rho}(a_s(x): s\in \II) \coloneqq  
\sup_{J \in \ZZ_+} \sup_{I \in  \mathfrak S_J(\II)}
\Big(\sum_{j \in [J]}  \big|a_{I_{j}}(x)-a_{I_{j-1}}(x) \big|^{\rho} \Big)^{\frac{1}{\rho}}.
\end{align*}
If needed, we use the convention that the supremum taken over the empty set is equal to $0$.

We refer to \cite{MSzW} for useful properties of oscillation and variation seminorms.

\section{Multiparameter variation inequality for abstract operators}\label{sec:var}

\subsection{Abstract variational theorem} Throughout this section  $(X, \mathcal B(X), \mu)$ is a $\sigma$-finite measure space and $k \in \ZZ_+$ is fixed. Suppose that $\{ \calH_s : s \in \RR^k_+\}$ is a family of operators acting on $L^p(X)$. Our goal is to obtain the following multiparameter variation inequality
\[
\big\| 
V^2\big(\calH_s f : s\in\RR^k_+\big)
\big\|_{L^p(X)}\lesssim_p \| f\|_{L^p(X)},\qquad f\in L^p(X), 
\]
for certain values of $p \in (1, \infty)$, under some general conditions imposed on $\calH_s$. 

Here a model example is $X=\RR^d$ with Lebesgue measure, and a signed measure $\sigma$ satisfying $\sigma(\RR^d) = 0$ together with some additional smoothness and scaling properties. In this context, for each $s \in \RR^k_+$ we set $\calH_s f = \sigma_s \ast f$, where $\sigma_s$ is the $L^1$ dilation of $\sigma$ with the scale parameter $(P_1(s), \dots, P_d(s))$, while each $P_i$ is a monomial of $k$ variables, that is, $P_i(s) = {s_1}^{\alpha_{i,1}} \cdots {s_k}^{\alpha_{i,k}}$ for some $\alpha_{i,1}, \dots, \alpha_{i,k} \in \NN$.

\begin{theorem}[Abstract oscillation/variation theorem]\label{thm:2}
Let $(X, \mathcal B(X), \mu)$ be a $\sigma$-finite measure space and take $1 \leq p_0 < p_1  \leq 2$. Given $k \in \ZZ_+$, let $\{\calH_s : s\in\RR_+^k\}$ be a~family of linear operators defined on $L^{p_0}(X) + L^{2}(X)$. Suppose that the following conditions are satisfied. 

\begin{enumerate}[label={\rm (C\arabic*)}]
	\item\label{cond:abs1}(Maximal estimate) The maximal operator 
	$\calH_{*}f =  \sup_{s \in \RR_+^k} \sup_{\abs{g} \leq \abs{f}} \abs{ \calH_s g}
	$
	is bounded on $L^{p_1}(X)$. \smallskip
	
	
	\item\label{cond:abs3} (Splitting into long and short variations) There exist parameters $L, r \in \ZZ_+$ and a family of cubes $ \{ Q_{\bf g} \subseteq \RR^k_+ : {\bf g} \in \ZZ^r \}$ for which the following is true. For each ${\bf g} \in \ZZ^r$ the cube $Q_{\bf g}$ is of the form 
	\[
	Q_{\bf g} = \big\{ s \in \RR^k_+ : {\rm left}(Q_{\bf g}) \preceq s \preceq {\rm right}(Q_{\bf g}) \big\}
	\]
	with some ${\rm left}(Q_{\bf g}), {\rm right}(Q_{\bf g}) \in \RR^k_+$ satisfying ${\rm right}(Q_{\bf g}) = (1+2^L) \odot {\rm left}(Q_{\bf g})$.
	Moreover, there are fixed parameters $s({\bf g}) \in  Q_{\bf g}$, ${\bf g} \in \ZZ^r$, for which the inequality
	\begin{equation*}
\qquad  	\big\| 
	V^2\big(\calH_s f : s\in\RR^k_+\big)
	\big\|_{L^{p_1}(X)} \lesssim \big\| V^2\big( \calH_{s({\bf g})} f : {\bf g} \in \ZZ^r \big)
	\big\|_{L^{p_1}(X)} + \Big\| \Big(
	\sum_{{\bf g} \in \ZZ^r} \big| V^2\big(\calH_{s} f : s\in Q_{\bf g} \big) \big|^2 \Big)^{\frac{1}{2}}
	\Big\|_{L^{p_1}(X)}
	\end{equation*}
	holds uniformly in $f \in L^{p_1}(X)$. \smallskip
	
	\item\label{cond:abs4} (Continuity)  Given $f \in L^{p_1}(X)$ and ${\bf g} \in \ZZ^r$ the mapping $s \mapsto \calH_s f (x)$ is continuous on $Q_{\bf g}$ for $\mu$-almost every $x \in X$. \smallskip
	
	\item\label{cond:abs5} (Littlewood--Paley decomposition) There exists a family of operators $ \{ \calS_{\bf g} : {\bf g} \in \ZZ^r \}$ commuting with each $\calH_s$, which plays the role of the family of Littlewood--Paley operators in the sense that 
	\begin{align}\label{eq:PL1}
	\sum_{{\bf g} \in \ZZ^r} \calS_{\bf g}= \operatorname{Id}
	\end{align}
	holds in the strong operator topology on $L^2(X)$ and $\mu$-almost everywhere on $X$ for every $f\in L^2(X)$.   Furthermore, assume that for every $p \in (p_0, 2]$ we have
	\begin{align}\label{eq:PL2}
	\norm[\Big]{\Big(\sum_{{\bf g} \in \ZZ^r} \abs{\calS_{\bf g} f}^{2} \Big)^{\frac{1}{2}}}_{L^p(X)}
	\lesssim_p
	\norm{f}_{L^p(X)}, \qquad f\in L^p(X).
	\end{align}
	Moreover, the operators $\calS_{\bf g}$ are adapted to the cubes $Q_{\bf g}$, that is, for each ${\bf g}, {\bf h} \in \ZZ^r$ and $s_{\bf g} \in Q_{\bf g}$
	\begin{align}\label{eq:off-diag}
	\Big\| \sum_{{\bf g} \in \ZZ^r} \calH_{s_{\bf g}} \calS_{{\bf g}+{\bf h}} f \Big\|_{L^{2}(X)} \leq c_{\bf h} \| f \|_{L^{2}(X)}
	\end{align}
	holds for some family $\{ c_{\bf h} : {\bf h} \in \ZZ^r \}\subseteq\RR_+$ and, additionally, for some $\varepsilon \in (0, \frac{1}{2})$ we have  
	\begin{align}\label{eq:sum-aj}
	\sum_{ {\bf h} \in \ZZ^r } c_{\bf h}^{(\frac{1}{2} - \varepsilon) \frac{p_1 - p_0}{2 - p_0}} < \infty. 
	\end{align}
	\item\label{cond:abs6} (Cancellation condition) For any $K\subseteq [k]$  (we allow $K=\emptyset$) we have a uniform estimate
	\begin{align*}
	\big\| \Delta_h^{K}\calH_{s} \big\|_{L^{p}(X) \rightarrow L^{p}(X)}\lesssim \prod\limits_{i \in K}\frac{h_i}{s_i} 
	\end{align*}
        for $p \in \{p_0,2\}$, whenever $s, s+h\in Q_{\bf g}$ and ${\bf g}\in\ZZ^r$.
\end{enumerate}	
Then the following $k$-parameter variation inequality holds
\begin{align}
\label{eq:3}
\big\| 
V^2\big(\calH_{s} f : s\in\RR^k_+\big)
\big\|_{L^{p_1}(X)} \lesssim \| f\|_{L^{p_1}(X)},\qquad f\in L^{p_1}(X), 
\end{align}
with the implicit constant possibly depending on $p_0, p_1, k, r, L, \varepsilon$, and the implicit constants arising in the assumptions.
In particular, we have the $k$-parameter oscillation inequality
\begin{align}
\label{eq:4}
\sup_{I \in \mathfrak S(\RR^k_+)}	\big\| O^2_I\big(\calH_{s} f : s\in\RR^k_+\big)
\big\|_{L^{p_1}(X)} \lesssim \| f\|_{L^{p_1}(X)},\qquad f\in L^{p_1}(X).
\end{align}
Additionally, the inequalities stated above hold with $p_1$ replaced by any $p \in [p_1, p_1']$ provided that the set of dual conditions is also satisfied, that is, \ref{cond:abs3} and \ref{cond:abs4} hold with $p_1$ replaced by every $p \in [p_1,p'_1]$, \ref{cond:abs1}
and \ref{cond:abs6} hold with the adjoint operators $\calH^*_s$ in place of $\calH_s$, and \eqref{eq:PL2} holds in the range $[2, p_0')$. 
\end{theorem}

Some remarks about Theorem~\ref{thm:2} are in order. 
\begin{enumerate}[label*={\arabic*}.]

\item The quantities arising in the inequality in  \ref{cond:abs3} are called respectively the long and short variations.

\item At the first glance it may be surprising that we estimate
$2$-variations \eqref{eq:3}, which, in view of L{\'e}pingle's
inequality, in many problems may be simply unbounded. But in Theorem
\ref{thm:2} we work with the family $\{\calH_s : s\in\RR_+^k\}$
obeying a very strong square function estimate \eqref{eq:off-diag},  which will allow us to handle long
$2$-variations on $L^2(X)$ and $L^{p_1}(X)$ and deduce the desired bounds.

\item In order to avoid some technicalities we have assumed that \eqref{eq:PL1} holds in the strong operator topology on $L^2(X)$ and $\mu$-almost everywhere on $X$ for every $f\in L^2(X)$. In fact, the latter assumption may be dropped and it only suffices to assume that  \eqref{eq:PL1} holds in the strong operator topology on $L^2(X)$ to obtain the conclusion of Theorem \ref{thm:2}. In practice, the usual Littlewood--Paley operators will be used for which \eqref{eq:PL1} holds in the strong operator topology on $L^2(X)$ and $\mu$-almost everywhere on $X$ for every $f\in L^2(X)$.

\item The dual conditions hold when $\calS_{\bf g}$ are self-adjoint and $\calH_s$ are convolution operators on a locally compact abelian group equipped with Haar measure. The model operators fit into this paradigm. 
\item For the sake of clarity we specified $Q_{\bf g}$ to be of a very particular simple form. However, many other variants are also available. For example, one can use ${\rm right}(Q_{\bf g}) = C_{\bf g} \odot {\rm left}(Q_{\bf g})$ with $C_{\bf g} \in (1,\infty)$, provided that the splitting in \ref{cond:abs3} holds and $\sup_{{\bf g} \in \ZZ^r} C_{\bf g} < \infty$. 
\item In Section~\ref{sec:rad} we always apply Theorem~\ref{thm:2} with $p_0=1$ although other choices are also possible. 
\end{enumerate}

The rest of this section will be devoted to proving
Theorem~\ref{thm:2}. It suffices to establish \eqref{eq:3} as
$2$-variations always dominate oscillations from \eqref{eq:4}. To
prove \eqref{eq:3} we will use \ref{cond:abs3} to estimate long and
short variations separately.  Since the space $X$ is fixed, we
shall abbreviate $L^p(X)$ to $L^p$.

\subsection{Littlewood--Paley operators} Condition \ref{cond:abs3}
allows one to divide the problem into two parts, global and local. In
both cases we gain enough control over the size of the scale
parameter~$s \in \RR^k_+$~to use efficiently properties of the
operators $\calS_{\bf g}$ from \ref{cond:abs5}.  Our first goal is to
use the off-diagonal decay \eqref{eq:off-diag} on $L^2(X)$ to
obtain a similar estimate on $L^{p_1}(X)$.  This will be done in
Lemma~\ref{lem:duo+rubio:squarem} below, but first we recall the
following result that derives a vector-valued inequality from a
maximal one.

\begin{lemma}[{cf.\ \cite[p.\ 544]{duo&rubio}}]
	\label{lem:duo+rubiom}
	Suppose that $(\calL_{v})_{v\in\VV}$ is a sequence of linear operators on $L^{1}(X)+L^{\infty}(X)$ indexed by a countable set $\VV$, and define
	$
	\calL_{*, \VV}f \coloneqq   \sup_{v\in\VV} \sup_{\abs{g} \leq \abs{f}} \abs{\calL_{v}g}.
	$
	Fix $q_{0},q_{1} \in [1,\infty]$ with $q_0 < \min\{2, q_1\}$ and take $\theta \in (0,1)$ such that $\frac12 = \frac{1-\theta}{q_{0}}$. Then for $q_{\theta}\in(q_0, q_1)$ given by
	$\frac{1}{q_{\theta}} = \frac{1-\theta}{q_{0}} + \frac{\theta}{q_{1}}$ we have
	\[
	\norm[\Big]{ \Big( \sum_{v\in\VV} \abs{\calL_{v}g_{v}}^{2}\Big)^{\frac{1}{2}}}_{L^{q_{\theta}}}
	\leq
	\big(\sup_{v\in\VV} \norm{\calL_{v}}_{L^{q_{0}}\to L^{q_{0}}}^{1-\theta}\big) \norm{\calL_{*, \VV}}_{L^{q_{1}}\to L^{q_{1}}}^{\theta}
	\norm[\Big]{ \Big( \sum_{v\in\VV} \abs{g_{v}}^{2}\Big)^{\frac{1}{2}}}_{L^{q_{\theta}}}.
	\]
\end{lemma}
\begin{proof}
The proof  runs along  the same lines as the proof of \cite[Lemma~2.8]{bootstrap}, we omit the details. 
\end{proof}

Next we use Lemma~\ref{lem:duo+rubiom} to obtain the following off-diagonal decay on $L^{p_1}(X)$.
 
\begin{lemma}
	\label{lem:duo+rubio:squarem}
	Let $p_{0}, p_{1}, r, \calS_{\bf g}$ be as in Theorem~\ref{thm:2} and for fixed $k_0\in\ZZ_+$, $\mathcal V \in\ZZ_+^{k_0}$ denote 
	\[
	\ZZ^r_{\mathcal V} \coloneqq  \Set[\big]{({\bf g}, v)\in\ \ZZ^r \times \ZZ_+^{k_0} : v \preceq {\mathcal V}}.
	\] 
	Suppose that $(\calL_{{\bf g},v})_{({\bf g}, v)\in\ZZ^r_{\mathcal V}}$ is a sequence of operators satisfying
	\begin{equation}
	\label{eq:duo+rubio:off-diagm}
	\norm[\Big]{ \Big( \sum_{({\bf g}, v)\in\ZZ^r_{\mathcal V}} \abs{ \calL_{{\bf g},v} \calS_{{\bf g}+{\bf h}}f}^{2} \Big)^{\frac{1}{2}} }_{L^2}
	\leq
	\tilde{c}_{{\bf h}} \norm{f}_{L^2}, \qquad f\in L^{2}(X), \ {\bf h} \in \ZZ^r,
	\end{equation}
	with some sequence of positive numbers $(\tilde{c}_{{\bf h}})_{{\bf h} \in \ZZ^r}$.
	Then for all $f\in L^{p_1}(X)$ we have
	\begin{align*}
	\norm[\Big]{\Bigl( \sum_{({\bf g}, v)\in\ZZ^r_{\mathcal V}} \abs{ \calL_{{\bf g},v} \calS_{{\bf g}+{\bf h}}f}^{2} \Bigr)^{\frac{1}{2}}}_{L^{p_1}}
	\lesssim (\mathcal V_1\cdots \mathcal V_{k_0})^{\frac{2-p_{1}}{2-p_{0}} \frac{1}{2}}
	\sup_{({\bf g}, v)\in\ZZ^r_{\mathcal V}} \norm{\calL_{{\bf g},v}}_{L^{p_{0}}\to L^{p_{0}}}^{\frac{2-p_{1}}{2-p_{0}} \frac{p_{0}}{2}} 
	\norm{\calL_{*, \ZZ^r_{\mathcal V}}}_{L^{p_{1}}\to L^{p_{1}}}^{\frac{2-p_{1}}{2}} \tilde{c}_{{\bf h}}^{\frac{p_{1}-p_{0}}{2-p_{0}}}
	\norm{f}_{L^{p_1}},
	\end{align*}
	where the implicit constants are numerical and in particular do not depend on ${\bf h} \in \ZZ^r$. 
\end{lemma}
\begin{proof} The statement is trivial for $p_1 = 2$, so we assume that $1 \leq p_0 < p_1 < 2$. We apply Lemma~\ref{lem:duo+rubiom} with $q_0 = p_0$, $q_1 = p_1$, and $\VV = \ZZ^r_{\mathcal V}$. Then by that lemma and \eqref{eq:PL2} we obtain
	\begin{align*}
	\norm[\Big]{\Bigl( \sum_{({\bf g}, v)\in\ZZ^r_{\mathcal V}} \abs{ \calL_{{\bf g},v} \calS_{{\bf g}+{\bf h}}f}^{2} \Bigr)^{\frac{1}{2}}}_{L^{p_{\theta}}} 
	&\lesssim
	\Big(\sup_{({\bf g}, v)\in\ZZ^r_{\mathcal V}} \norm{\calL_{{\bf g},v}}_{L^{p_{0}}\to L^{p_{0}}}^{1-\theta}\Big) \norm{\calL_{*, \ZZ^r_{\mathcal V}}}_{L^{p_{1}}\to L^{p_{1}}}^{\theta}
	\norm[\Big]{\Bigl( \sum_{({\bf g}, v)\in\ZZ^r_{\mathcal V}} \abs{ \calS_{{\bf g}+{\bf h}}f}^{2} \Bigr)^{\frac{1}{2}}}_{L^{p_{\theta}}}\\
	&\lesssim
	(\mathcal V_1 \cdots \mathcal V_{k_0})^{\frac{1}{2}} \sup_{({\bf g}, v)\in\ZZ^r_{\mathcal V}} \norm{\calL_{{\bf g},v}}_{L^{p_{0}}\to L^{p_{0}}}^{1-\theta} \norm{\calL_{*, \ZZ^r_{\mathcal V}}}_{L^{p_{1}}\to L^{p_{1}}}^{\theta}
	\norm{f}_{L^{p_{\theta}}},
	\end{align*}
	where $\theta = 1 - \frac{p_0}{2}$ and $\frac{1}{p_\theta} = \frac{1-\theta}{p_0} + \frac{\theta}{p_1}$. Interpolating this with the hypothesis \eqref{eq:duo+rubio:off-diagm} gives the claim.
\end{proof}

\subsection{Long variations} The long variations from \ref{cond:abs3} will be controlled if we prove a stronger result, that is, the following square function inequality
\[
\Big\| \Big( \sum_{{\bf g} \in \ZZ^r} \big| \calH_{s({\bf g})} f \big|^2 \Big)^{\frac{1}{2}}
\Big\|_{L^{p_1}} \lesssim \| f \|_{L^{p_1}}.
\]
Note that \eqref{eq:PL1} implies
\begin{equation}\label{eq:PL_commutation}
\Big\| \Big( \sum_{{\bf g} \in \ZZ^r} \big| \calH_{s({\bf g})} f \big|^2 \Big)^{\frac{1}{2}}
\Big\|_{L^{p_1}} \leq 
\sum_{{\bf h} \in \ZZ^r} \Big\| \Big( \sum_{{\bf g} \in \ZZ^r} \big| \calH_{s({\bf g})} \calS_{{\bf g}+{\bf h}} f \big|^2 \Big)^{\frac{1}{2}}
\Big\|_{L^{p_1}}.
\end{equation}
Then \eqref{eq:off-diag} with $s_{\bf g} = s(\bf g)$ enables us to apply Lemma~\ref{lem:duo+rubio:squarem}. Indeed, we take $k_0=1$, $\mathcal V=1$, $\calL_{{\bf g},1}=\calH_{s({\bf g})}$, and $\tilde{c}_{\bf h} = c_{\bf h}$, and then, combining Lemma~\ref{lem:duo+rubio:squarem} with \eqref{eq:PL_commutation}, we obtain
\begin{equation*}
\Big\| \Big( \sum_{{\bf g} \in \ZZ^r} \big| \calH_{s({\bf g})} f \big|^2 \Big)^{\frac{1}{2}}
\Big\|_{L^{p_1}} \lesssim \sup_{s \in \RR^k_+} \| \calH_s \|_{L^{p_0} \to L^{p_0}}^{\frac{2-p_1}{2-p_0} \frac{p_0}{2}} \| \calH_{*} \|_{L^{p_1}\to L^{p_1}}^{\frac{2-p_1}{2}} \sum_{{\bf h} \in \ZZ^r} c_{\bf h}^{\frac{p_1-p_0}{2 - p_0}}
\| f \|_{L^{p_1}}.
\end{equation*}
In view of \ref{cond:abs1}, \eqref{eq:sum-aj}, and \ref{cond:abs6} for $K = \emptyset$, the quantities in the last expression are finite so we are done.

\subsection{Multiparameter Rademacher--Menshov  inequality}
We are left with estimating the short variations from \ref{cond:abs3}. Our next goal is to bring the cancellation condition from \ref{cond:abs6} into play. In the one-parameter case ($k=1$) an important tool to control short variations is the Rademacher--Menshov  inequality, which asserts   
that for a  given $L \in \ZZ_+$ and a family $\{a_s : s \in [1, 1+2^{L}] \}\subseteq\CC$ we have
\begin{equation}\label{eq:num2}
V^2\big(a_s : s \in [1, 1+2^{L}]\big) \leq \sqrt{2} \sum_{l \in \ZZ_+} \Big( 
\sum_{j \in [2^l]} \big| a_{1+j2^{L-l}} - a_{1+(j-1)2^{L-l}} \big|^2
\Big)^{\frac{1}{2}},
\end{equation}
provided that the map $s \mapsto a_s$ is continuous. The proof of this inequality can be found in \cite[Lemma~2.5]{bootstrap}, see also \cite[Lemma~2.1]{BMSW1}. Our aim is to prove a multiparameter variant of \eqref{eq:num2}.

 Let $\mathfrak{a} = \{ a_{{s}} : {s} \in [1, 1+2^{L}]^{k_0} \}\subseteq \CC$ be a~collection of numbers such that the map ${s} \mapsto a_{s}$ is a~continuous function of $k_0$ variables, and additionally the following condition is satisfied
\begin{equation}\label{vanishing}
a_{{s}} = 0 \text{ for each } s = (s_1, \dots, s_{k_0}) \text{ such that } 1 \in \{ s_1, \dots, s_{k_0}\}.
\end{equation} 
Then the multiparameter Rademacher--Menshov  inequality holds
\begin{equation}\label{eq:num}
V^2 \big(a_{{s}} : {s} \in [1, 1+2^{L}]^{k_0} \big) \leq 
2^{\frac{k_0}{2}} \sum_{l\in\ZZ_+^{k_0}} \Big( 
\sum_{j\in[{\bf 2}^l]}\big| \Delta_{2^{L} \odot {\bm 2}^{-l}}^{[k_0]}(a_{{\bf 1} + 2^{L} \odot {\bm 2}^{-l} \otimes (j-\bf 1)}) \big|^2
\Big)^{\frac{1}{2}},
\end{equation}
where we slightly abuse our notation and write $[{\bf 2}^l] \coloneqq [2^{l_1}] \times \cdots \times [2^{l_{k_0}}]$.
We refer also to \cite[Lemma 8.1]{KMT}.
\begin{proof}[Proof of inequality \eqref{eq:num}]
For clarity, we only present the proof for $k_0=2$, the general case is a simple elaboration of the ideas for $k_0=2$. Assume that $\mathfrak{a} = \{ a_{s_1, s_2} : (s_1, s_2) \in [1, 1+2^{L}]^2 \}\subseteq\CC$ is a collection of numbers such that the map $(s_1, s_2) \mapsto a_{s_1, s_2}$ is continuous as a function of two variables. Given $(1,1) \preceq (s_1,s_2) \prec (s_1',s_2') \preceq (1+2^{L}, 1+2^{L})$ we want to estimate the quantity $|a_{s'_1, s'_2} - a_{s_1, s_2}|$. By the continuity we may assume that  $(1,1) \prec (s_1, s_2)$. Note that
\begin{align*}
a_{s'_1, s'_2} - a_{s_1, s_2} & = \big( a_{s'_1, s'_2} - a_{s_1, s'_2} - a_{s'_1, s_2} + a_{s_1, s_2} \big) + \big( a_{s_1, s'_2} - a_{1, s'_2} - a_{s_1, s_2} + a_{1, s_2} \big) \\
& \qquad + \big( a_{s'_1, s_2} - a_{s_1, s_2} - a_{s'_1, 1} + a_{s_1, 1} \big) + \big( a_{1, s'_2} - a_{1, s_2} + a_{s'_1, 1} - a_{s_1, 1} \big),
\end{align*}
where $( a_{1, s'_2} - a_{1, s_2} + a_{s'_1, 1} - a_{s_1, 1})=0$ by 
condition \eqref{vanishing}.

In other words, the quantity $a_{s'_1, s'_2} - a_{s_1, s_2}$ can be rewritten as a combination of difference expressions corresponding to the disjoint rectangles $[s_1, s'_1) \times [s_2, s'_2)$, $[1, s_1) \times [s_2, s'_2)$, and $[s_1,s'_1) \times [1, s_2)$.

Moreover, observe that if $(1,1) \prec (s_1,s_2) \prec (s_1',s_2') \prec (s_1'',s_2'') \preceq (1+2^{L}, 1+2^{L})$, then the rectangles used to rewrite $a_{s'_1, s'_2} - a_{s_1, s_2}$ and $a_{s_1'', s_2''} - a_{s'_1, s'_2}$, respectively, are also disjoint (see Figure~\ref{fig:1}). 

\begin{figure}[ht]
	\begin{tikzpicture}[scale=0.8,
	axis/.style={very thick, ->, >=stealth'}, 	
	important line/.style={thick}, 
	dashed line/.style={dashed, thin}, 
	]
	
	\draw[axis] (0,0)  -- (6,0) node(xline)[right] {$s_1$};
	\draw[axis] (0,0) -- (0,6) node(yline)[left] {$s_2$};
	
	\draw (0.0,0.0) node[below] {$0$} -- (0.0,0.0);
	\draw (0,1) node[left] {$1$} -- (0.1,1);
	\draw (1,0) node[below] {$1$} -- (1,0.1);
	\draw (0,5) node[left] {$1+2^{L}$} -- (0.1,5);
	\draw (5,0) node[below] {$1+2^{L}$} -- (5,0.1);
	
	\draw[axis] (9,0)  -- (15,0) node(xline)[right] {$s_1$};
	\draw[axis] (9,0) -- (9,6) node(yline)[left] {$s_2$};
	
	\draw (9.0,0.0) node[below] {$0$} -- (9.0,0.0);
	\draw (9,1) node[left] {$1$} -- (9.1,1);
	\draw (10,0) node[below] {$1$} -- (10,0.1);
	\draw (9,5) node[left] {$1+2^{L}$} -- (9.1,5);
	\draw (14,0) node[below] {$1+2^{L}$} -- (14,0.1);
	
	\node[label={[below right]}, style={circle,fill,inner sep=2pt}] at (2,2) {};
	\node[label={[below right]}, style={circle,fill,inner sep=2pt}] at (3,3.5) {};
	\node[label={[below right]}, style={circle,fill,inner sep=2pt}] at (4.5,4.5) {};
	
	\node[label={[above right]$-$}, style={circle,inner sep=0pt}] at (1.9,1.85) {};
	\node[label={[below left]$+$}, style={circle,inner sep=0pt}] at (3.05,3.55) {};
	
	\node[label={[above right]$-$}, style={circle,inner sep=0pt}] at (2.9,3.35) {};
	\node[label={[below left]$+$}, style={circle,inner sep=0pt}] at (4.55,4.55) {};
	
	\node[style={circle,fill,inner sep=2pt}] at (11,2) {};
	\node[style={circle,fill,inner sep=2pt}] at (12,3.5) {};
	\node[style={circle,fill,inner sep=2pt}] at (13.5,4.5) {};
	
	\draw[dashed line] (1,0) -- (1,6);
	\draw[dashed line] (0,1) -- (6,1);
	
	\draw[dashed line] (10,0) -- (10,6);
	\draw[dashed line] (9,1) -- (15,1);
	
	\draw[important line] (2,2) -- (2,3.5);
	\draw[important line] (3,2) -- (3,3.5);
	\draw[important line] (3,3.5) -- (3,4.5);
	\draw[important line] (4.5,3.5) -- (4.5,4.5);
	
	\draw[important line] (2,2) -- (3,2);
	\draw[important line] (2,3.5) -- (3,3.5);
	\draw[important line] (3,3.5) -- (4.5,3.5);
	\draw[important line] (3,4.5) -- (4.5,4.5);
	
	\draw[dashed line] (11,3.5) -- (11,2);
	\draw[dashed line] (12,4.5) -- (12,3.5);
	
	\draw[dashed line] (12,2) -- (11,2);
	\draw[dashed line] (13.5,3.5) -- (12,3.5);
	
	\draw[important line] (11,2) -- (11,1);
	\draw[important line] (11,2) -- (10,2);
	
	\draw[important line] (12,3.5) -- (12,1);
	\draw[important line] (12,3.5) -- (10,3.5);
	
	\draw[important line] (13.5,4.5) -- (13.5,1);
	\draw[important line] (13.5,4.5) -- (10,4.5);
	
	\node[label={[below right]}, style={circle,fill,inner sep=1pt}] at (11,1) {};
	\node[label={[below right]}, style={circle,fill,inner sep=1pt}] at (12,1) {};
	\node[label={[below right]}, style={circle,fill,inner sep=1pt}] at (13.5,1) {};
	
	\node[label={[below right]}, style={circle,fill,inner sep=1pt}] at (10.0,2) {};
	\node[label={[below right]}, style={circle,fill,inner sep=1pt}] at (12,2) {};
	
	\node[label={[below right]}, style={circle,fill,inner sep=1pt}] at (10,3.5) {};
	\node[label={[below right]}, style={circle,fill,inner sep=1pt}] at (11,3.5) {};
	\node[label={[below right]}, style={circle,fill,inner sep=1pt}] at (13.5,3.5) {};
	
	\node[label={[below right]}, style={circle,fill,inner sep=1pt}] at (10,4.5) {};
	\node[label={[below right]}, style={circle,fill,inner sep=1pt}] at (12,4.5) {};
	
	\node[label={[above right]$+$}, style={circle,inner sep=0pt}] at (10.9,1.85) {};
	\node[label={[above left]$-$}, style={circle,inner sep=0pt}] at (11.1,1.85) {};
	\node[label={[below right]$-$}, style={circle,inner sep=0pt}] at (10.9,2.05) {};
	
	\node[label={[below left]$+$}, style={circle,inner sep=0pt}] at (12.1,3.55) {};
	\node[label={[below right]$-$}, style={circle,inner sep=0pt}] at (11.9,3.55) {};
	\node[label={[above right]$+$}, style={circle,inner sep=0pt}] at (11.9,3.35) {};
	\node[label={[above left]$-$}, style={circle,inner sep=0pt}] at (12.1,3.35) {};
	
	\node[label={[above right]$+$}, style={circle,inner sep=0pt}] at (9.9,1.85) {};
	\node[label={[above left]$-$}, style={circle,inner sep=0pt}] at (10.1,1.85) {};
	
	\node[label={[above right]$+$}, style={circle,inner sep=0pt}] at (9.9,3.35) {};
	\node[label={[above left]$-$}, style={circle,inner sep=0pt}] at (10.1,3.35) {};
	\node[label={[below right]$-$}, style={circle,inner sep=0pt}] at (9.9,3.55) {};
	\node[label={[below left]$+$}, style={circle,inner sep=0pt}] at (10.1,3.55) {};
	
	\node[label={[below right]$-$}, style={circle,inner sep=0pt}] at (10.9,3.55) {};
	\node[label={[below left]$+$}, style={circle,inner sep=0pt}] at (11.1,3.55) {};
	
	\node[label={[below right]$-$}, style={circle,inner sep=0pt}] at (9.9,4.55) {};
	\node[label={[below left]$+$}, style={circle,inner sep=0pt}] at (10.1,4.55) {};
	
	\node[label={[below right]$-$}, style={circle,inner sep=0pt}] at (11.9,4.55) {};
	\node[label={[below left]$+$}, style={circle,inner sep=0pt}] at (12.1,4.55) {};
	
	\node[label={[above left]$-$}, style={circle,inner sep=0pt}] at (12.1,1.85) {};
	\node[label={[below left]$+$}, style={circle,inner sep=0pt}] at (12.1,2.05) {};
	
	\node[label={[above left]$-$}, style={circle,inner sep=0pt}] at (12.1,0.85) {};
	\node[label={[below left]$+$}, style={circle,inner sep=0pt}] at (12.1,1.05) {};
	\node[label={[above right]$+$}, style={circle,inner sep=0pt}] at (11.9,0.85) {};
	\node[label={[below right]$-$}, style={circle,inner sep=0pt}] at (11.9,1.05) {};
	
	\node[label={[above left]$-$}, style={circle,inner sep=0pt}] at (13.6,0.85) {};
	\node[label={[below left]$+$}, style={circle,inner sep=0pt}] at (13.6,1.05) {};
	
	\node[label={[above left]$-$}, style={circle,inner sep=0pt}] at (13.6,3.35) {};
	\node[label={[below left]$+$}, style={circle,inner sep=0pt}] at (13.6,3.55) {};
	
	\node[label={[above right]$+$}, style={circle,inner sep=0pt}] at (10.9,0.85) {};
	\node[label={[below right]$-$}, style={circle,inner sep=0pt}] at (10.9,1.05) {};
	
	\node[label={[below left]$+$}, style={circle,inner sep=0pt}] at (13.6,4.55) {};
	
	\draw[->, distance=3pt,thick] (6.50,3) -- (8.0,3);
	
	\end{tikzpicture}
	\caption{Symbolic display of rewriting $a_{s'_1, s'_2} - a_{s_1, s_2}$ and $a_{s_1'', s_2''} - a_{s'_1, s'_2}$ by using difference expressions associated with disjoint rectangles. The three thick dots in each coordinate system represent the three points $(1,1) \prec (s_1,s_2) \prec (s_1',s_2') \prec (s_1'',s_2'') \preceq (1+2^{L}, 1+2^{L})$.}
	\label{fig:1}
\end{figure}
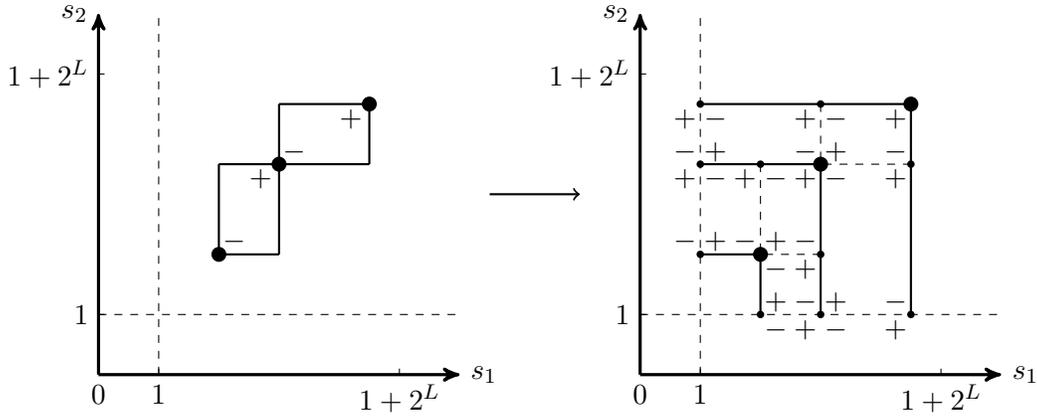

 Next, each of these rectangles can be decomposed into a family of disjoint rectangles of the form 
\[
\big[1+(j_1-1)2^{L-l_1} , 1+j_12^{L-l_1}  \big) \times \big[ 1+(j_2-1)2^{L-l_2} , 1+j_22^{L-l_2}\big)
\] 
and for any fixed pair $(l_1, l_2) \in \ZZ^2_+$ there are at most four rectangles of side lengths $2^{L-l_1}$ and $2^{L-l_2}$, respectively, used in the decomposition. By these considerations we are led to the following inequality
\[
V^2 \big(a_{s_1, s_2} : (s_1, s_2) \in [1, 1+2^{L}]^2 \big) \leq 
2 \sum_{l_1 \in \ZZ_+} \sum_{l_2\in \ZZ_+} \Big( 
\sum_{j_1 \in [2^{l_1}]} \sum_{j_2 \in [2^{l_2}]} \big| \Delta_{l_1,l_2,j_1,j_2}(\mathfrak{a}) \big|^2
\Big)^{\frac{1}{2}},
\]
where $\Delta_{l_1,l_2,j_1,j_2}(\mathfrak{a})$ stands for the expression
\[
a_{1+j_12^{L-l_1}, 1+j_22^{L-l_2}} - a_{1+(j_1-1)2^{L-l_1}, 1+j_22^{L-l_2}} - a_{1+j_12^{L-l_1}, 1+(j_2-1)2^{L-l_2}} + a_{1+(j_1-1)2^{L-l_1}, 1+(j_2-1)2^{L-l_2}}.
\]
This completes the proof of inequality \eqref{eq:num}.
\end{proof}

\subsection{Vanishing property} Given ${\bf g} \in \ZZ^r$ and $x \in X$ we would like to estimate $V^2( \calH_{s} f(x) : s\in Q_{\bf g})$ using \eqref{eq:num}. Unfortunately,   the collection $\{ \calH_{s} f(x) : s\in Q_{\bf g} \}$, a priori, may not have the vanishing property as in \eqref{vanishing}. To overcome this difficulty we introduce auxiliary projective measures.

For each ${\bf g} \in \ZZ^r$ and $K \subseteq [k]$ consider $\pi^{{\bf g},K} \colon Q_{\bf g} \rightarrow \RR^k_+$ given by
\[
\big(\pi^{{\bf g},K}(s)\big)_i\coloneqq  \begin{cases}
s_i, &\quad i\in [k] \setminus K,\\
\big({\rm left}(Q_{\bf g})\big)_i, & \quad i\in K. 
\end{cases} 
\]
Then for any $s \in Q_{\bf g}$, ${\bf g} \in \ZZ^r$, and $K \subseteq [k]$ we define the following projective operator
\[
\Pi_{s}^{{\bf g},K} \coloneqq   \calH_{\pi^{{\bf g},K}(s)}.
\]
 Applying the inclusion--exclusion principle we obtain
$\calH_s=\sum_{K\subseteq[k]}\calH_s^{{\bf g},K}$, where
\[
\calH^{{\bf g},K}_s\coloneqq  \sum_{K'\subseteq[k]\setminus K}(-1)^{|K'|}\Pi_s^{{\bf g}, K\cup K'}.
\]
In view of this decomposition the short variations from \ref{cond:abs3} are controlled as follows
\begin{equation}\label{eq:inex}
\Big\| \Big(
\sum_{{\bf g} \in \ZZ^r} \big| V^2\big(\calH_{s} f : s\in Q_{\bf g} \big) \big|^2 \Big)^{\frac{1}{2}}
\Big\|_{L^{p_1}} \lesssim \sum_{K \subseteq [k]} \Big\| \Big(
\sum_{{\bf g} \in \ZZ^r} \big| V^2\big(\calH_s^{{\bf g},K} f : s\in Q_{\bf g} \big) \big|^2 \Big)^{\frac{1}{2}} \Big\|_{L^{p_1}}.
\end{equation}

A few remarks about the properties of $\calH_s^{{\bf g},K}$ are in order. Firstly, let us emphasize that $\calH^{{\bf g}, K}_s = \calH^{{\bf g}, K}_{\pi^{{\bf g},K}(s)}$ holds for each $s \in Q_{\bf g}$. Thus, given $K \subseteq [k]$ we can think of $V^2\big(\calH_s^{{\bf g},K} f : s\in Q_{\bf g} \big)$ as the variation seminorm with $k - |K|$ parameters, that is, those corresponding to the set $[k] \setminus K$, since the parameters corresponding to $K$ are frozen. In particular, the summand corresponding to $K = [k]$ in \eqref{eq:inex} equals $0$. Moreover, if there exists $i \in [k] \setminus K$ such that $s_i = ({\rm left}(Q_{\bf g}))_i$, then $\calH^{{\bf g}, K}_s$ is the zero operator. Thus, the vanishing property for $\calH^{{\bf g}, K}_s$ is exactly what we need in order to apply \eqref{eq:num} with $k_0 = k - |K|$.

\subsection{Short variations} Now to prove Theorem~\ref{thm:2} it suffices to estimate the quantity 
\[
\Big\| \Big(
\sum_{{\bf g} \in \ZZ^r} \big| V^2\big(\calH_s^{{\bf g},K} f : s\in Q_{\bf g} \big) \big|^2 \Big)^{\frac{1}{2}} \Big\|_{L^{p_1}}
\]
for each $K \subsetneq [k]$. To keep our notation simple, we assume that $K = [k] \setminus [k_0]$ for $k_0 \in \ZZ_+$ with $k_0 \leq k$ (we have already excluded the case $k_0=0$ corresponding to $K = [k]$). With this choice of $K$ we define
\[
\calH^{{\bf g},K}_{l,j} f \coloneqq   \Delta_{{\bf h}_{{\bf g}, l, L}}^{[k_0]}(\calH^{{\bf g},K}_{{\bf s}_{{\bf g}, l, L, j}}) f, \qquad {\bf g} \in \ZZ^r, \
l \in \ZZ_+^{k_0},\ j\in [{\bf 2}^l], 
\]
where 
\[
{\bf s}_{{\bf g}, l, L, j}\coloneqq {\rm left}(Q_{\bf g}) \otimes \big(1+ (j_1-1)2^{L-l_1}, \dots, 1+ (j_{k_0}-1)2^{L-l_{k_0}}, 1, \dots, 1\big),
\]
and
\[
{\bf h}_{{\bf g}, l, L}\coloneqq {\rm left}(Q_{\bf g}) \otimes \big(2^{L-l_1} , \dots, 2^{L-l_{k_0}}, 0, \dots, 0\big).
\]
Applying \eqref{eq:num} with $a_{{s}} = \calH^{{\bf g}, K}_{{\rm left}(Q_{\bf g}) \otimes (s_1, \dots, s_{k_0}, 1, \dots, 1)} f(x)$ we may write
\[
\Big\| \Big(
\sum_{{\bf g} \in \ZZ^r} \big| V^2\big(\calH_s^{{\bf g},K} f : s\in Q_{\bf g} \big) \big|^2 \Big)^{\frac{1}{2}} \Big\|_{L^{p_1}} \leq 
2^{\frac{k_0}{2}} \sum_{l\in \ZZ_+^{k_0}} \Big\| \Big(\sum_{{\bf g} \in \ZZ^r} \sum_{j \in [{\bf 2}^l]} \big| \calH^{{\bf g},K}_{l,j} f \big|^{2}\Big)^{\frac{1}{2}} \Big\|_{L^{p_1}},
\]
since ${s} \mapsto \calH^{{\bf g}, K}_{{\rm left}(Q_{\bf g}) \otimes (s_1, \dots, s_{k_0}, 1, \dots, 1)} f(x)$ is continuous on $[1, 1+2^{L}]^{k_0}$ by \ref{cond:abs4} for $\mu$-almost every $x\in X$.

Now proceeding as in \eqref{eq:PL_commutation}, we obtain
\[
\Big\| \Big(
\sum_{{\bf g} \in \ZZ^r} \big| V^2\big(\calH_s^{{\bf g},K} f : s\in Q_{\bf g} \big) \big|^2 \Big)^{\frac{1}{2}} \Big\|_{L^{p_1}} \leq 2^{\frac{k_0}{2}} 
\sum_{{\bf h} \in \ZZ^r}\sum_{l\in \ZZ_+^{k_0}} \calI^{p_1}({\bf h},l),
\]
where for $q\in [1,\infty)$ the quantity $\calI^q({\bf h},l)$ is defined by
\[
\calI^q({\bf h},l)\coloneqq  \Big\| \Big(\sum_{{\bf g}\in\ZZ^r} \sum_{j \in [{\bf 2}^l]} \big| \calH^{{\bf g},K}_{l,j} \calS_{{\bf g}+{\bf h}} f)  \big|^{2}\Big)^{\frac{1}{2}}\Big\|_{L^q}.
\]
Observe that the family $\{ \calH^{{\bf g},K}_s : s \in Q_{\bf g} \}$ satisfies \ref{cond:abs6} with $[k]$ replaced by $[k_0]$. 
Thus, for $p \in \{p_0,2\}$,
\begin{equation} \label{eq:shortholder}
\| \calH^{{\bf g},K}_{l,j} \|_{L^p\to L^p} \lesssim 2^{-|l|_1}
\end{equation}
and we deduce from \eqref{eq:off-diag} and \eqref{eq:shortholder} that the assumptions of Lemma~\ref{lem:duo+rubio:squarem} are satisfied, since for each fixed $l \in \ZZ_+^{k_0}$, taking $\mathcal V = {\bf 2}^{l}$ and $\calL_{{\bf g},v} = \calH^{{\bf g},K}_{l,v}$, inequality \eqref{eq:duo+rubio:off-diagm} holds with
$\tilde{c}_{\bf h}= C 2^{|l|_1/2} \min\{2^{-|l|_1}, c_{\bf h} \},$
for $C \in \RR_+$ large enough. In other words, we have
\[
\calI^2({\bf h},l) \lesssim 2^{|l|_1/2} \min\{2^{-|l|_1}, c_{\bf h} \} \|f\|_{L^2}.
\]
Let us now observe that \ref{cond:abs1} easily implies that the maximal operator corresponding to the  family $\{ \calH^{{\bf g},K}_s : {\bf g} \in \ZZ^r, \, s \in Q_{\bf g} \}$ is bounded on $L^{p_1}(X)$. By \eqref{eq:shortholder} for $p=p_0$, we see that Lemma~\ref{lem:duo+rubio:squarem} gives 
\[
\calI^{p_1}({\bf h},l) \lesssim
\big( 2^{|l|_1/2} \min\{2^{-|l|_1}, c_{\bf h} \} \big)^{\frac{p_1-p_0}{2-p_0}}\|f\|_{L^{p_1}}.
\]
It remains to notice that by \eqref{eq:sum-aj} the ultimate expression is summable in ${\bf h}\in\ZZ^r$ and $l \in \ZZ_+^{k_0}$.

This completes the proof of Theorem \ref{thm:2}.

\section{Multiparameter oscillation inequality for Radon operators}\label{sec:rad}
In this section we will use Theorem~\ref{thm:2} to prove Theorem~\ref{thm:3} stating that for every $p\in(1, \infty)$ we have
\begin{align}
\label{eq:5}
\sup_{I \in \mathfrak S(\RR^k_+)}
\big\| O^2_I\big(\calM^\PPP_{s} f : s\in\RR^k_+\big)
\big\|_{L^{p}(\RR^{d})} \leq C \| f\|_{L^{p}(\RR^{d})}, \qquad f\in L^{p}(\RR^{d}),
\end{align}
whenever $\PPP=(P_1,\ldots, P_d) \colon \RR^k \to \RR^d$ is a polynomial mapping such that each $P_i$ is a monomial, that is, $P_i(s) = {s_1}^{\alpha_{i,1}} \cdots {s_k}^{\alpha_{i,k}}$ for some $\alpha_{i,1}, \dots, \alpha_{i,k} \in \NN$. We prove \eqref{eq:5} by induction on $d$. The proof will be divided into several steps, which will take up the bulk of this section.

\subsection{Associated matrix}
We fix a polynomial mapping
$\PPP=(P_1,\ldots, P_d) \colon \RR^k \to \RR^d$.  From now on we
assume that each $P_i$ is a monomial as above. We introduce the
$d \times k$ matrix $\mathbb{M}^\PPP$ whose rows $R_i$ are the vectors
$(\alpha_{i,1}, \dots, \alpha_{i,k})$. We may assume that
$\mathbb{M}^\PPP$ has no zero columns and rows. Indeed, if
$\mathbb{M}^\PPP$ has a zero column, say, the $k$-th one, then $s_k$
is not used in any of the polynomials $P_i$. Thus, one can just remove
the integration over this parameter from the definition of
$\calM^\PPP_s$ and hence the $k$-th column of $\mathbb{M}^\PPP$ can be
removed as well. On the other hand, if $\mathbb{M}^\PPP$ has a zero
row, say, the $d$-th one, then  one can think that the
$d$-th space variable is frozen. Then it turns out that the expected
inequality is a consequence of a related inequality stated for a
lower-dimensional space. Below we present a version of this fact.

\begin{lemma}
	\label{lem:2m}
	Given $d_0,k_0 \in \ZZ_+$ and $p \in (1, \infty)$ if a family $\{ T_{\RR^{d_0}}[\Xi_s] : s \in \RR^{k_0}_+\}$ satisfies the inequality
	\begin{equation*}
	\big \| O^2_I \big(T_{\RR^{d_0}}[\Xi_s]  f : s \in \RR^{k_0}_+ \big) \big \|_{L^p(\RR^{d_0})} \leq C \| f \|_{L^p(\RR^{d_0})}, \qquad f \in L^p(\RR^{d_0}),
	\end{equation*}
	for some $I \in \mathfrak S(\RR^{k_0}_+)$ and $C \in \RR_+$, then for each $i \in [d_0+1]$ the family $\{ T_{\RR^{d_0+1}}[\Xi_s^i] : s \in \RR^{k_0}_+\}$ 
	given by
	\[
	\Xi_s^i(\xi) \coloneqq  \Xi_s(\xi_1, \dots, \xi_{i-1},\xi_{i+1}, \dots, \xi_{d_0+1}), \qquad \xi = (\xi_1, \dots, \xi_{d_0+1}) \in \RR^{d_0+1},
	\]
	satisfies
	\begin{equation*}
	\big \| O^2_I \big( T_{\RR^{d_0+1}}[\Xi_s^i] f : s \in \RR^{k_0}_+ \big) \big\|_{L^p(\RR^{d_0+1})} \leq C \| f \|_{L^p(\RR^{d_0+1})}
	\end{equation*}
	for each $ f \in L^p(\RR^{d_0+1})$ such that the integral defining the norm on the left hand side exists. The same result holds for the maximal operators in place of the oscillation seminorms. 
\end{lemma}

\noindent We emphasize that each time we will apply Lemma~\ref{lem:2m} there are no measurability problems.

\begin{proof}
	We focus on the oscillation seminorm case and the proof for the maximal operator is analogous. 
	By symmetry it suffices to consider $i=d_0+1$. Hence, $\Xi_s^{d_0+1}(\xi,\xi_{d_0+1})=\Xi_s(\xi)$. For $f \in L^p(\RR^{d_0+1})$ we let $f_y(x) \coloneqq  f(x,y)$, $x\in\RR^{d_0}$, $y\in\RR$. Using H\"older's inequality twice we obtain
	\begin{align*}
	&\Big| \int_{\RR} \int_{\RR^{d_0}} O^2_I\big(T_{\RR^{d_0+1}}[\Xi_s^{d_0+1}] f : s\in\RR^{k_0}_+\big)(x,y)g(x,y)\,\dif x\,\dif y\Big| \\
	&\qquad =\Big| \int_{\RR} \int_{\RR^{d_0}} O^2_I\big(T_{\RR^{d_0}}[\Xi_s] f_y : s\in\RR^{k_0}_+\big)(x)g_y(x)\,\dif x\,\dif y\Big|\\
	&\qquad \leq \int_{\RR} \big\| O^2_I\big(T_{\RR^{d_0}}[\Xi_s] f_y : s\in\RR^k_+\big)\big\|_{L^p(\RR^{d_0})} \| g_y\|_{L^{p'}(\RR^{d_0})}\,\dif y\\
	&\qquad \leq C \int_{\RR} \|  f_y\|_{L^p(\RR^{d_0})} \| g_y\|_{L^{p'}(\RR^{d_0})}\,\dif y\\
	&\qquad \leq C \| f\|_{L^p(\RR^{d_0+1})} \| g\|_{L^{p'}(\RR^{d_0+1})}.
	\end{align*}
	Taking the supremum over functions $g$ such that $\| g\|_{L^{p'}(\RR^{d_0+1})}=1$ we obtain the claim. 
\end{proof}

\subsection{Decomposition} Now we have to decompose $\calM_s^\PPP$ in such a way that either Theorem~\ref{thm:2} or the induction hypothesis can be applied to estimate each of the resulting components. 
Since $\nu_s^\PPP(\RR^d) = 1$, the corresponding multiplier $\mathfrak m_s^\PPP$ does not vanish near the origin and
Theorem~\ref{thm:2} cannot be applied directly to the averaging operators $\calH_s=\calM_s^\PPP$. To overcome this problem we will proceed more delicately. 

 Given a fixed Schwartz function $\Phi \colon \RR\rightarrow \RR_+$ such that $\int_{\RR}\Phi(x) \, \dif x=1$ we let $\Upsilon \coloneqq  \calF_\RR \Phi$. Then for each $D \subseteq [d]$ and $s \in \RR^k_+$ we define $\mathfrak m_s^{\PPP, D} \colon \RR^d \to \CC$ by
 \[
 \mathfrak m_s^{\PPP, D}(\xi) \coloneqq  \Big( \prod_{i \in D} \Upsilon \big(P_i(s) \xi_i\big) \Big) \mathfrak m_s^{\PPP}\big( \pi^D(\xi) \big),
 \] 
 where $\pi^D(\xi)\in\RR^d$ satisfies $(\pi^D(\xi))_i = 0$ if $i \in D$ and $(\pi^D(\xi))_i = \xi_i$ otherwise. The operator $T_{\RR^d}[\mathfrak m_s^{\PPP, D}]$ may be seen as a composition of two operators acting on different variables. Precisely, one of them is an operator of Radon type associated with a ``shorter'' polynomial mapping consisting of $d - |D|$ monomials, while the other one is a convolution with some smooth function of $|D|$ variables. 

Clearly, we have $\mathfrak m^{\PPP, \emptyset}_s = \mathfrak m^\PPP_s$ and 
\begin{equation} \label{eq:mudecomp}
\mathfrak m^\PPP_s = \sum_{D \subseteq [d] : D \neq \emptyset} (-1)^{|D|+1} \mathfrak m^{\PPP, D}_s+\tilde{\mathfrak m}^\PPP_s,
\end{equation}
where $\tilde{\mathfrak m}^\PPP_s \coloneqq   \sum_{D \subseteq [d]} (-1)^{|D|} \mathfrak m^{\PPP, D}_s$. Thus, it suffices to estimate uniformly in $I \in \mathfrak S(\RR^k_+)$ the quantity
\begin{equation*}
\big\|O^2_I\big(T_{\RR^d}[\tilde{\mathfrak m}^\PPP_s] f : s\in\RR_+^k\big)\big\|_{L^p(\RR^d)}+  \sum_{D \subseteq [d] : D \neq \emptyset} \big\|O^2_I\big(T_{\RR^d}[\mathfrak m^{\PPP, D}_s]f : s\in\RR_+^k\big)\big\|_{L^p(\RR^d)}.
\end{equation*}
To bound the first quantity we will use Theorem \ref{thm:2}, whereas the second quantity will be controlled by the inductive argument.
The advantage of this decomposition is that the following estimate
\begin{equation} \label{cond:2}
|\tilde{\mathfrak m}_s^\PPP(\xi)|\lesssim \prod_{i =1}^d \min \big\{ |P_i(s) \xi_i|^{\delta}, |P_i(s) \xi_i|^{-\delta} \big\},\qquad \xi\in\RR^d,
\end{equation} 
holds for some small $\delta \in \RR_+$. In order to see this observe that by the inclusion--exclusion principle we have
$
\tilde{\mathfrak m}^\PPP_s = \sum_{D \subseteq [d]} \tilde{\mathfrak m}^{\PPP,D}_s,
$
where $\tilde{\mathfrak m}^{\PPP,D}_s \colon \RR^d \to \CC$ is given by
\[
\tilde{\mathfrak m}^{\PPP,D}_s(\xi)\coloneqq  \Big( \prod_{i \in D} \Upsilon \big( P_i(s) \xi_i\big) \Big) \Big( \prod_{i \in [d] \setminus D } 1 - \Upsilon \big( P_i(s) \xi_i\big)  \Big) \Big( \square^D \mathfrak m_s^\PPP (\xi) \Big)
\]
and
\[
\square^D \mathfrak m_s^\PPP(\xi)\coloneqq  \sum_{D' \subseteq D} (-1)^{|D'|} \mathfrak m_s^{\PPP}\big( \pi^{D'}(\xi) \big).
\]
Then, using the following observations:
\begin{enumerate}[itemsep=0.2cm, label=(\roman*)]
	\item a version of the van der Corput estimate for nonsmooth amplitude functions implies $|\square^D \mathfrak m_s^\PPP(\xi)| \lesssim |P_i(s) \xi_i|^{-\delta_0}$ for $i \notin D$ with some $\delta_0 \in \RR_+$ depending on $\PPP$ (see \cite[Proposition~B.2]{bootstrap});
	\item the trivial estimate $|\bm e(x) - 1| \lesssim |x|$ gives $|\square^D \mathfrak m_s^\PPP(\xi)| \lesssim \min \{ 1, | P_i(s) \xi_i| \}$ for $i \in D$;
	\item since $\Phi$ is a Schwartz function of integral $1$,  we have $|\Upsilon( P_i(s) \xi_i)| \lesssim \min\{1, | P_i(s) \xi_i|^{-1}\}$ for $i \in D$, and $|1 - \Upsilon(P_i(s) \xi_i)| \lesssim \min\{1, | P_i(s) \xi_i|\}$, for $i \not\in D$;
\end{enumerate}
we conclude that
\begin{align*}
| \tilde{\mathfrak m}^{\PPP,D}_s(\xi) | \lesssim \prod_{i =1}^d \min \big\{ |P_i(s) \xi_i|^{\delta}, |P_i(s) \xi_i|^{-\delta} \big\} 
\end{align*}
holds with some $\delta \in \RR_+$ for every $\xi \in \RR^d$ and $D \subseteq [d]$. Consequently, inequality \eqref{cond:2} is satisfied.

\subsection{Oscillation inequality for bumps} \label{sec:bumps} Before we bound the family $\{ T_{\RR^d}[\tilde{\mathfrak m}^\PPP_s] : s \in \RR^k_+\}$ we take a closer look at the multipliers $\mathfrak m^{\PPP, D}_s$ that arise in the decomposition in \eqref{eq:mudecomp}. As mentioned before, these objects are dual in nature and, loosely speaking, the smaller the Radon part is, the less singular the operator becomes. Since for $D = [d]$ the Radon part disappears completely, this case should be the starting point of our analysis. Below we prove the oscillation inequality for the family $\{ T_{\RR^d}[\mathfrak m^{\PPP,[d]}_s] : s \in \RR^k_+\}$. 

Let us begin with the simplest case $d=1$. We first notice that for $k=1$ and $P = {\rm id}$ the inequality
\begin{equation}\label{eq:bumps}
\sup_{I \in\mathfrak{S}(\RR_+)}
\norm[\big]{O_I^2 \big(T_\RR[\mathfrak m_s^{{\rm id}, [1]}] f : s \in \RR_+ \big)}_{L^p(\RR)}
\lesssim_p
\norm{f}_{L^p(\RR)}, \qquad f\in L^p(\RR),
\end{equation}
is well known (see \cite{jsw} or \cite{MSzW} for the recent approach). 

 All the other cases with $d=1$ are direct consequences of \eqref{eq:bumps}. Indeed, if $k \in \ZZ_+$ and $P$ is an arbitrary nonconstant monomial of $k$ variables whose coefficient equals $1$, then we can just substitute $S = P(s)$. Since $s \prec s'$ implies $P(s) < P(s')$, for each $I \in \mathfrak S(\RR^k_+)$ we obtain the equality 
\[
\big\|O^2_I\big(T_\RR[\mathfrak m_s^{P, [1]}] f : s\in\RR_+^k\big)\big\|_{L^p(\RR)}
=
\big\|O^2_{I^P}\big(T_\RR[\Upsilon_{S}] f : S \in\RR_+\big)\big\|_{L^p(\RR)},
\]
where $\Upsilon_{S}(\xi) \coloneqq  \calF_\RR \Phi(S\xi)$ 
and $I^P \coloneqq   \{ P(I_j) : j \in \NN \} \in \mathfrak S(\RR_+)$.
Thus,
\[
\sup_{I \in \mathfrak S(\RR^k_+)} \big\|O^2_I\big(T_\RR[\mathfrak m_s^{P, [1]}] f : s\in\RR_+^k\big)\big\|_{L^p(\RR)}
\leq
\sup_{I \in \mathfrak S(\RR_+)} \big\|O^2_I\big(T_\RR[\Upsilon_{S}] f : S \in\RR_+\big)\big\|_{L^p(\RR)}  
\]
and we are reduced to the case $k=1$ and $P = {\rm id}$ treated before.

Finally, observe that the general case $d \in \NN$ follows by multiple application of the one-dimensional result. For each $D \subseteq [d]$ and $s \in \RR^k_+$ denote
\begin{align*}
\Upsilon^{\PPP,D}_s(\xi) \coloneqq   \prod_{i \in D} \Upsilon(P_i(s)\xi_i),
\qquad \xi = (\xi_1, \dots, \xi_d) \in \RR^d. 
\end{align*}
Then for a given $s^\circ \in \RR_+^k$ and an arbitrary set $\mathbb{S} \subseteq \RR_+^k$ we have
	\begin{align*}
	\sup_{s \in \mathbb{S}} \, \big\vert 
	T_{\RR^d} \big[ \mathfrak m_s^{\PPP, [d]}  - \mathfrak m_{s^\circ}^{\PPP, [d]} \big] f \big\vert
	& = \sup_{s \in \mathbb{S}} \, \big\vert 
	T_{\RR^d} \big[ \Upsilon^{\PPP,[d]}_s - \Upsilon^{\PPP,[d]}_{s^\circ} \big] f \big\vert \\
	& \leq \sum_{i \in [d]} \, \sup_{s \in \mathbb{S}} \, \big\vert T_{\RR^d} \big[ \Upsilon^{\PPP,[i-1]}_{s} \big( \Upsilon^{\PPP,\{i\}}_{s} - \Upsilon^{\PPP,\{i\}}_{s^\circ} \big) \Upsilon^{\PPP,[d] \setminus [i]}_{s^\circ} \big] f \big\vert \\
	& \lesssim \sum_{i \in [d]} \, \calM_{\ast}^{(1)}\cdots \calM_{\ast}^{(i-1)} \calM_{\ast}^{(i+1)} \cdots \calM_{\ast}^{(d)} \Big( \sup_{s \in \mathbb{S}} \,
	\big\vert T_{\RR^d} \big[ \Upsilon^{\PPP,\{i\}}_{s} - \Upsilon^{\PPP,\{i\}}_{s^\circ} \big] f \big\vert
	\Big),
	\end{align*}
	where $\calM_{\ast}^{(i)}$ denotes the Hardy--Littlewood maximal operator acting on the $i$-th variable. Hence, invoking the Fefferman--Stein vector-valued inequality for each $\calM_{\ast}^{(i)}$ we see that
	\[
	\big\|O^2_I\big(T_{\RR^d} [ \mathfrak m_s^{\PPP, [d]}] f : s\in\RR_+^k\big)\big\|_{L^p(\RR^d)}
	\lesssim_{p,\PPP} \sum_{i \in [d]}
	\big\|O^2_I\big(T_{\RR^d} \big[ \Upsilon^{\PPP,\{i\}}_{s} \big] f : s\in\RR_+^k\big)\big\|_{L^p(\RR^d)}
	\]
	holds uniformly in $I \in \mathfrak S(\RR^k_+)$. This combined with the one-dimensional result and Lemma~\ref{lem:2m} yields 
\begin{equation}\label{thm:9}	
	\sup_{I \in \mathfrak S(\RR^k_+)}
	\big\| O_I^2 \big(T_{\RR^d} [ \mathfrak m_s^{\PPP, [d]}] f:s\in\RR^k_+ \big) \big\|_{L^p(\RR^d)}
	\lesssim_{p,\PPP}
	\norm{f}_{L^p(\RR^d)}, \qquad f\in L^p(\RR^d),
\end{equation}
which finishes the proof in the general case.
\subsection{Splitting into long and short variations}  
Our next goal is to prove the following inequality
\begin{equation}\label{eq:sigma}
\big\| V^2\big(T_{\RR^d}[\tilde{\mathfrak m}_s^\PPP(\xi)] f : s\in\RR_+^k\big)\big\|_{L^p(\RR^{d})}\lesssim_{p,\PPP} \|f\|_{L^p(\RR^{d})}
\end{equation}
by using Theorem~\ref{thm:2}. Inequality \eqref{eq:sigma}, in the case $d=1$, follows from \eqref{eq:mudecomp} and \eqref{thm:9}. Before we prove it for general $d\in\ZZ_+$ we have to establish \ref{cond:abs3} and \ref{cond:abs5} with $\calH_s=T_{\RR^d}[\tilde{\mathfrak m}_s^\PPP]$. 

Let us begin with \ref{cond:abs3}. In the one-parameter case it is known that
\begin{align}
\label{eq:6}
V^2(a_s : s \in \RR_+) \lesssim V^2(a_s : s \in \DD) + \Big(\sum_{n \in \ZZ} V^2\big(a_s : s \in [2^n, 2^{n+1})\big)^2 \Big)^{\frac{1}{2}}
\end{align}
holds with the implicit constant independent of $\{ a_s : s \in \RR_+ \} \subseteq \CC$. Now we have to prove 
a multiparameter variant of \eqref{eq:6}, which is a fairly complicated task due to the fact that
the map $s \mapsto \PPP(s)$ may neither be injective nor surjective. Here we present an argument, which is the core of our paper.

Let $r \in \ZZ_+$ be the rank of $\mathbb{M}^\PPP$. Of course, $r \leq \min\{d,k\}$. We choose a set $\{i_1, \dots, i_r\} \subseteq [d]$ such that the corresponding rows $R_{i_1}, \dots, R_{i_r}$ are linearly independent. It is straightforward to check that the map $\PPP^r \colon \RR^k_+ \rightarrow \RR^r_+$ defined by $s \mapsto (P_{i_1}(s), \dots, P_{i_r}(s))$ is surjective. Moreover, for each $s \in \RR^k_+$ the value of $\PPP(s)$ is completely determined by the value of $\PPP^r(s)$, since each row in $\mathbb{M}^\PPP$ can be expressed as a linear combination of vectors $R_{i_1}, \dots, R_{i_r}$. For each $n \in \ZZ^r$ we pick some $s(n) \in \RR^k_+$ such that $\PPP^r(s(n)) = {\bf 2}^n \in \DD^r$ and the long variations will be defined with the aid of the set $\{ s(n) : n \in \ZZ^r\}$. 

Next, we focus on the short variations. We construct 
a family of cubes $\{ Q_n \subseteq \RR^k_+ : n \in \ZZ^r\}$ not necessarily disjoint whose union covers $\RR^k_+$ and such that the elements of $Q_n$ are comparable with $s(n)$.

For each $N \in \ZZ_+$ consider $E_N \coloneqq   \{  \PPP^r(s) : s \in [2^{-N}, 2^{N}]^k \} \cap [1,2]^r$, which is a closed subset of $[1,2]^r$ with its natural topology. By Baire's theorem there exists $E_{N}$ with nonempty interior. Then, exploiting the fact that $\PPP^r$ is surjective, together with the relation $\PPP^r(s \otimes u) = \PPP^r(s) \otimes \PPP^r(u)$, one can show that
\begin{equation}\label{eq:Q1}
  [1,2]^r \subseteq \big\{  \PPP^r (s) : s \in [2^{-N_0}, 2^{N_0}]^k \big\}
\end{equation}
holds for sufficiently large $N_0 \in \ZZ_+$. Assuming that $N_0 \geq 2$ we also have
\begin{equation}\label{eq:Q2}
\Big( \min\{ s_i : i \in [k] \} \geq 1 \ \wedge \ \max\{ s_i : i \in [k] \} \geq 2^{N_0} \Big) \implies \Big( \PPP^r(s) \notin [1,2]^r \Big), 
\end{equation}
since $\mathbb{M}^\PPP$ has no zero columns. We will show that the choice
\begin{equation} \label{eq:Qn}
Q_n \coloneqq   \big\{s \in \RR^k_+ : 2^{-2 N_0} \odot s(n) \preceq s \preceq (2^{-2 N_0}(1+2^{4N_0})) \odot s(n)\big\}, \qquad n \in \ZZ^r,
\end{equation}
will be good for our purposes (thus, $4N_0$ can play the role of the parameter $L$ from Section~\ref{sec:var}). 

\begin{lemma}\label{lem:splitting}
Given $k, d \in \ZZ_+$ let  $\PPP=(P_1,\ldots, P_d) \colon \RR^k \to \RR^d$ be a polynomial map such that each $P_i$, $i\in [d]$,  is a nonconstant monomial with coefficient $1$. Define $r,s(n),Q_n$ as before. Then the inequality 

\begin{equation}\label{eq:splitting}
V^2(a_s : s \in \RR^k_+) \lesssim
V^2(a_{s(n)}  : n \in \ZZ^r ) + \Big(
\sum_{n \in \ZZ^r} \big( V^2(a_{s} : s \in Q_n) \big)^2 \Big)^{\frac{1}{2}}
\end{equation}
holds uniformly for families $\{a_s : s \in \RR^k_+\}\subseteq \CC$ of numbers satisfying $\PPP(s)=\PPP(s') \implies a_s = a_{s'}$.
\end{lemma}
\begin{proof}
	Fix $J \in \ZZ_+$ and let $(I_i : i \in \{0\} \cup [J]) \in \mathfrak S_J(\RR^k_+)$. Our goal is to estimate
	\begin{equation}\label{eq:seq}
	\Big( \sum_{j \in [J]} \big| a_{I_{j}} - a_{I_{j-1}} \big|^2 \Big)^{\frac{1}{2}}.
	\end{equation}
	For each $j \in \{ 0 \} \cup [J]$ we let $n(j)$ be the unique parameter $n \in \ZZ^r$ for which $\PPP^r (I_j) \in R(n)$, where $R(n) \coloneqq  \{ \tau \in \RR^r_+ : {\bf 2}^n \preceq \tau \prec {\bf 2}^{n + {\bf 1}}\}$. Observe that $n(0) \preceq \cdots \preceq n(J)$. 
	Suppose that $n(i-1) \neq n(i) = \cdots = n(i') \neq n(i'+1)$ holds for some $i,i' \in [J-1]$ with $i \leq i'$. Then for $j \in \{i-1, i, i'+1\}$, since $\PPP^r (I_j) \in R(n(j))=\PPP^r (s(n(j)))\otimes[1, 2]^r$, we can use \eqref{eq:Q1} to find 
	\[
	I^*_j \in \big\{s \in \RR^k_+ : 2^{- N_0} \odot s(n(j)) \preceq s \preceq 2^{N_0} \odot s(n(j))\big\} \subseteq Q_{n(j)}
	\] 
	such that $\PPP^r(I^*_j) = \PPP^r(I_j)$ and, consequently, $a_{I^*_j} = a_{I_j}$. Next, for all $j \in [i'] \setminus [i]$ we set $I^*_j = I^*_i \otimes I_i^{-1} \otimes I_j$. With this choice we have $\PPP^r(I^*_j) = \PPP^r(I_j)$ and $I^*_i \prec \cdots \prec I^*_{i'}$. Moreover, $\PPP^r(I_i^{-1} \otimes I_j)\in[1, 2]^r$ thus by \eqref{eq:Q2}
we obtain $|I_i^{-1} \otimes I_j|_{\infty}\le 2^{N_0}$. Then  it is not difficult to see that        the conditions $n(i) = n(i')$ and $2^{- N_0} \odot s(n(i)) \preceq I^*_i  \preceq 2^{N_0} \odot s(n(i))$ imply $\{I_i^*, \dots, I^*_{i'} \} \subseteq Q_{n(i)}$. Thus, we have
	\begin{align*}
	\big|a_{I_i} - a_{I_{i-1}}\big|^2 &= \big|a_{I^*_i} - a_{I^*_{i-1}}\big|^2 
    \\
	&\lesssim \big| a_{s(n(i))} - a_{s(n(i-1))} \big|^2 + \sum_{j \in \{i-1,i\}} \big| a_{s(n(j))} - a_{{\rm left}(Q_{n(j)})} \big|^2 + \big| a_{I^*_j} - a_{{\rm left}(Q_{n(j)})}\big|^2 \\
	&\lesssim \big| a_{s(n(i))} - a_{s(n(i-1))} \big|^2 + \sum_{j \in \{i-1,i\}} \big( V^2(a_{s} : s \in Q_{n(j)} ) \big)^2,
	\end{align*}
	where ${\rm left}(Q_{n(j)})\coloneqq  2^{-2N_0} \odot s(n(j))$ is the ``bottom left'' corner of $Q_{n(j)}$. Similarly,
	\[ 
	\big| a_{I_{i'+1}} - a_{I_{i'}} \big|^2 =
	\big| a_{I^*_{i'+1}} - a_{I^*_{i'}} \big|^2 \lesssim \big| a_{s(n(i'+1))} - a_{s(n(i'))} \big|^2 + \sum_{j \in\{i',i'+1\} } \big( V^2 (a_{s} : s\in Q_{n(j)} ) \big)^2.
	\]
	Finally, we also have
	\[
	\sum_{j \in [i'] \setminus [i]} \big| a_{I_j} - a_{I_{j-1}} \big|^2 = \sum_{j \in [i'] \setminus [i]} \big| a_{I^*_j} - a_{I^*_{j-1}} \big|^2 \leq \big( V^2 (a_{s} : s\in Q_{n(i)}) \big)^2.
	\]
	Taking into account the inequalities above we see that \eqref{eq:seq} is controlled by the right hand side of \eqref{eq:splitting}. Since the implied constant does not depend on $J \in \ZZ_+$ and $I_0 \prec \cdots \prec I_J$, we conclude the claim.  
\end{proof}

\noindent Applying Lemma~\ref{lem:splitting} we obtain that the
left hand side of \eqref{eq:sigma} is dominated uniformly by
\[
\big\| V^2\big(T_{\RR^d}[\tilde{\mathfrak m}^\PPP_{s(n)}] f : n \in \ZZ^r \big)
\big\|_{L^p(\RR^{d})} + \Big\| \Big(
\sum_{n \in \ZZ^r}  V^2\big(T_{\RR^d}[\tilde{\mathfrak m}^\PPP_{s}] f : s \in Q_n \big)^2 \Big)^{\frac{1}{2}}
\Big\|_{L^p(\RR^{d})}.
\]

\subsection{Littlewood--Paley operators} Having obtained  condition \ref{cond:abs3}, we now prove \ref{cond:abs5} for suitably chosen operators $\calS_n$, $n \in \ZZ^r$, adapted to the cubes $Q_n$ defined in \eqref{eq:Qn}. 
For $n \in \ZZ^r$ let $\calS_n$ be the Littlewood--Paley projection corresponding to the cube $R(-n)$, that is, we set $\calS_n f \coloneqq   T_{\RR^d} [\Xi_n] f$, where 
\[
\Xi_n ( \xi_1, \dots, \xi_d) \coloneqq  \ind{R(-n)} (|\xi_{i_1}|, \dots, |\xi_{i_r}|).
\]
Note that the family $\{\calS_n : n \in \ZZ^r\}$ satisfies \ref{cond:abs5}. Indeed, \eqref{eq:PL1} and \eqref{eq:PL2} with any $p \in (1,2]$ follow from the general Littlewood--Paley theory. Moreover, in view of \eqref{cond:2}, for each $n, n' \in \ZZ^r$ and $\xi \in \RR^d$ such that $(|\xi_{i_1}|, \dots, |\xi_{i_r}|) \in R(-n-n')$ we have
\[
|\tilde{\mathfrak m}_{s(n)}^\PPP(\xi) | \lesssim \min \big\{ |2^{n_1} \xi_{i_1}|^{\pm \delta}, \dots, |2^{n_r} \xi_{i_r}|^{\pm  \delta} \big\} \lesssim 2^{- \frac{ \delta |n'|_1}{r}}. 
\]
Referring to the homogeneity properties of $\PPP$ we deduce that $s(n)$ can be replaced by any other element of $Q_n$, that is,  
$
| \tilde{\mathfrak m}_{s_n}^\PPP(\xi)  | \lesssim 2^{- \frac{ \delta |n'|_1}{r}}
$
holds for each $s_n \in Q_n$ and $\xi \in \RR^d$ such that $(|\xi_{i_1}|, \dots, |\xi_{i_r}|) \in R(-n-n')$ (we allow the implicit constant to depend on $\PPP$ and $N_0$). Consequently, taking $c_{n'} = C 2^{- \frac{\delta |n'|_1}{r}}$ with some large $C \in \RR_+$, we conclude that \eqref{eq:off-diag} and \eqref{eq:sum-aj} are satisfied as well. Thus, condition \ref{cond:abs5} is justified. 

\subsection{Verification of the remaining conditions in Theorem~\ref{thm:2}} To prove \eqref{eq:sigma}, we need to show that the family $\{T_{\RR^d}[\tilde{\mathfrak m}_{s}^\PPP] : s \in \RR_+^k \}$ satisfies the assumptions of Theorem~\ref{thm:2}.
We set $p_0 = 1$, $p_1 = p$, and $\calH_s = T_{\RR^d}[\tilde{\mathfrak m}_{s}^\PPP]$ and focus temporarily on the case $p \in (1,2]$. We have to verify conditions \ref{cond:abs1}--\ref{cond:abs6}. In the proof we will invoke induction hypothesis
if $d > 1$.  Thus, we assume that $d=1$ or that $d > 1$ and for each family of Radon operators acting on $d_0$-dimensional space with $d_0 < d$ the conclusion  of Theorem~\ref{thm:3} is satisfied.


Conditions \ref{cond:abs3} and \ref{cond:abs5} have already been proved, while \ref{cond:abs4} follows from the dominated convergence theorem. It remains to verify \ref{cond:abs1} and \ref{cond:abs6}. 

In order to justify \ref{cond:abs1}, we  show that the following maximal operator
\[
\calH_{*}f \coloneqq   \sup_{s \in \RR_+^k} \sup_{\abs{g} \leq \abs{f}} \abs{T_{\RR^d}[\tilde{\mathfrak m}_{s}^\PPP] g}
\]
is bounded on $L^p(\RR^d)$. We can assume that $f\geq 0$, since $\calH_{*}f=\calH_{*}|f|$. Observe that for $N \in \NN$ we have
\begin{equation}\label{eq:sigmaestimate}
\calH_{*,N}f \coloneqq   \sup_{s \in Q_{\le N}} \sup_{\abs{g} \leq \abs{f}} \abs{T_{\RR^d}[\tilde{\mathfrak m}_{s}^\PPP] g} \leq \sum_{D \subseteq [d] : D \neq \emptyset} \sup_{s \in Q_{\le N}} T_{\RR^d}[\mathfrak m_{s}^{\PPP,D}] f + \sup_{s \in Q_{\le N}}T_{\RR^d}[\mathfrak m_{s}^\PPP] f,
\end{equation}
where $Q_{\le N}\coloneqq \bigcup_{|n|_{\infty}\le N}Q_n$  and $Q_n$ has been defined in \eqref{eq:Qn}. Note that both $T_{\RR^d}[\mathfrak m_{s}^{\PPP,D}]$ and  $T_{\RR^d}[\mathfrak m_{s}^\PPP]$ are positive operators.
The last term in \eqref{eq:sigmaestimate} is uniformly controlled by
\[
\sup_{|n|_\infty \leq N} T_{\RR^d}[\mathfrak m_{{\rm right}(Q_n)}^\PPP] f \leq \sum_{D \subseteq [d] : D \neq \emptyset} \sup_{s \in \RR^k_+} T_{\RR^d}[\mathfrak m_{s}^{\PPP,D}] f + \Big( \sum_{|n|_\infty \leq N} \big| T_{\RR^d}[\tilde{\mathfrak m}_{{\rm right}(Q_n)}^\PPP] f \big|^2\Big)^{\frac{1}{2}}.
\]
Fix a nonempty $D \subseteq [d]$. We claim that for every $p\in(1, \infty)$ the following estimate holds
\begin{align}
\label{eq:7}
\big\| \sup_{s \in \RR_+^k} T_{\RR^d}[\mathfrak m_{s}^{\PPP,D}] f \big\|_{L^p(\RR^{d})} \lesssim_{p, \PPP} \| f \|_{L^p(\RR^{d})}.
\end{align}
Indeed, inequality \eqref{eq:7} for $D=[d]$ follows from \eqref{thm:9} (alternatively, we can use the $L^p$ boundedness of the Hardy--Littlewood maximal operator). We are left with the case $d > 1$ and $\emptyset \neq D \subseteq [d]$. By symmetry we can assume that $D = [d-d_0]$ for some $d_0 \in [d-1]$. Then
\[
\sup_{s \in \RR_+^k}T_{\RR^d}[\mathfrak m_{s}^{\PPP,D}] f \lesssim \calM_{\ast}^{(1)}\cdots \calM_{\ast}^{(d-d_0)} \sup_{s \in \RR_+^k} T_{\RR^d}[\mathfrak m_{s}^{\PPP}\circ \pi^D]f,
\]
yields the claim in \eqref{eq:7} by using Lemma~\ref{lem:2m} combined with  $L^p$ boundedness of the Hardy--Littlewood maximal operators and the inductive hypothesis
ensuring $L^p$ boundedness of the corresponding maximal Radon operators acting on $\RR^{d_0}$ (since the oscillation seminorm dominates the maximal one, see \eqref{eq:domin}).

Now we apply Lemma~\ref{lem:duo+rubio:squarem} with $k_0=1$, $\mathcal V = 1$, $q_0 = 1$, $q_1 = p$, $\calL_{n,0} = T_{\RR^d}[\tilde{\mathfrak m}_{{\rm right}(Q_n)}^\PPP]$ for $|n|_{\infty} \leq N$ and $\calL_{n,0} \equiv 0$ otherwise, and $\calS_n$ introduced before, to obtain 
\[
\Big\| \Big( \sum_{|n|_\infty \leq N} \big| T_{\RR^d}[\tilde{\mathfrak m}_{{\rm right}(Q_n)}^\PPP] f \big|^2\Big)^{\frac{1}{2}} \Big\|_{L^p(\RR^{d})} \lesssim_{p, \PPP} \| \calH_{*, N} \|_{L^p(\RR^{d}) \to L^p(\RR^{d})}^{\frac{2-p}{p}} \| f \|_{L^p(\RR^{d})},
\]
where \ref{cond:abs5} has been used to justify $\sum_{n \in \NN^r} \tilde{c}_n^{\ p-1} \lesssim_{p,\PPP} 1$. Consequently,
\[
\| \calH_{*, N} \|_{L^p(\RR^{d}) \to L^p(\RR^{d})} \lesssim_{p,\PPP} 1 + \| \calH_{*, N} \|_{L^p(\RR^{d}) \to L^p(\RR^{d})}^{\frac{2-p}{p}}
\]
and, since the implicit constant does not depend on $N$, we conclude that $\calH_{*}$ is bounded on $L^p(\RR^d)$ by the monotone convergence theorem. As a consequence of the previous arguments  the maximal operator corresponding to the family $\{ T_{\RR^d}[\mathfrak m_{s}^\PPP] : s \in \RR^k_+ \}$ is bounded on $L^p(\RR^d)$ for $p \in (1, 2]$.
Arguing as in \eqref{eq:sigmaestimate} and using simple interpolation for the maximal operators associated with the families $\{ T_{\RR^d}[\mathfrak m_{s}^\PPP] : s \in \RR^k_+ \}$ and $\{ T_{\RR^d}[\mathfrak m_{s}^{\PPP,D}] : s \in \RR^k_+ \}$ we deduce 
$L^p$ boundedness of $\calH_{*}$ for all $p \in (1, \infty)$, which gives \ref{cond:abs1}.

\medskip

Finally, we show that condition \ref{cond:abs6} holds. To see this we need some auxiliary results.

\begin{lemma}\label{lem:tensor}
	Given $d_1, d_2 \in \ZZ_+$ let $\{\calH_s^{(1)} : s \in \RR^k_+\}$ and $\{\calH_s^{(2)} : s \in \RR^k_+\}$ be two families of operators on $\RR^{d_1}$ and $\RR^{d_2}$, respectively, which satisfy condition \ref{cond:abs6}. Then the family $\{\calH_s^{(1)} \otimes \calH_s^{(2)} : s \in \RR^k_+\}$ of tensor products of operators on $\RR^{d_1+d_2}$ also satisfies \ref{cond:abs6} (the implicit constant may depend on $k$). 
\end{lemma}

\begin{proof}
	Fix $K \subseteq [k]$ and take $s, h \in \RR^k_+$ as in the statement of \ref{cond:abs6}. We decompose
	\[
	\Delta^{K}_h\big(\calH_{s}^{(1)} \otimes \calH_{s}^{(2)}\big)
	= \sum_{U \subseteq K} \Delta^{U}_h\calH_{s}^{(1)}\otimes \mathrm{T}_h^U\Delta^{K\setminus U}_h\calH_{s}^{(2)}
	\]
	and observe that each term has $L^2(\RR^{d_1 + d_2})$ operator norm  controlled by
	$
\prod_{i \in U} \frac{h_i}{s_i} \cdot \prod_{i \in K \setminus U} \frac{h_i}{s_i}  = \prod_{i \in K} \frac{h_i}{s_i}.
	$
	Thus, the family $\{\calH_s^{(1)} \otimes \calH_s^{(2)} : s \in \RR^k_+\}$ satisfies \ref{cond:abs6}.
\end{proof}

\begin{lemma}\label{lem:bump}
	Let $P:\RR^k_+\rightarrow \RR_+$ be a monomial of $k$ parameters whose coefficient is $1$. For $s \in \RR^k_+$ let $\calH^P_s f \coloneqq  \sigma^P_s \ast f$, where $\sigma^P_s$ is the measure on $\RR$ with the density $\rho_s(x) \coloneqq \frac{1}{P(s)} \Phi(\frac{x}{P(s)})$. Then the family $\{\calH^P_s : s \in \RR^k_+\}$ satisfies \ref{cond:abs6}.
\end{lemma}

\begin{proof}
Fix $\emptyset \neq K \subseteq [k]$ (the case $K = \emptyset$ is obvious) and $s, h \in \RR^k_+$ as in \ref{cond:abs6}. 
 By Young's inequality,
	\begin{equation*}
	\| \Delta^K_h \calH^P_{s} \|_{L^p(\mathfrak{\RR}) \rightarrow L^p(\mathfrak{\RR})}\leq \|\Delta^K_h \rho_s( \cdot) \|_{L^1(\RR)}.
	\end{equation*}
	By homogeneity it suffices to get the bound $C\prod_{i \in K} h_i$ for $s = {\bf 1}$. 
 Let $a_s \coloneqq P(s)$ and
 note that
	$
	\mathrm{T}^{K'}_h a_s = \prod_{i \in K'} (1+h_i)^{\alpha_i} \leq C_{P,N_0}
	$
	with some $C_{P,N_0} \in (1,\infty)$ and $\alpha_1, \dots, \alpha_k \in \ZZ_+$ depending on $P$. Thus,
	\begin{equation*}
	\Delta^K_h \rho_s(x) = \sum_{K' \subset K} (-1)^{|K'|} \frac{1}{\mathrm{T}^{K'}_h a_s} \Phi\Big(\frac{x}{\mathrm{T}^{K'}_h a_s}\Big) = \sum_{K' \subset K} (-1)^{|K'|} \Theta_x \big( \ln (\mathrm{T}^{K'}_h a_s ) \big),
	\end{equation*}
    where
	$
	\Theta_x(y) \coloneqq \frac{1}{e^y} \Phi(\frac{x}{e^y})$.
	Denote $\tilde{h}_i \coloneqq \alpha_i \ln(1+h_i)$. By applying the Newton--Leibniz formula $|K|$ times,
	\begin{equation*}
	\Big| \sum_{K' \subset K} (-1)^{|K'|} \Theta_x\Big( \sum_{i \in K'} \tilde{h}_i\Big) \Big| \leq \prod_{i \in K} \tilde{h}_i \cdot \max\big\{|\Theta_x^{(|K|)}(y)| : 0 \leq y \leq k \ln(C_{P,N_0}) \big\},
	\end{equation*}
	where $\Theta_x^{(|K|)}$ is the derivative of order $|K|$ of $\Theta$. Since $\prod_{i \in K} \tilde{h}_i \lesssim_{P,N_0} \prod_{i \in K} h_i$ and the maximum above is controlled by $(1+x^2)^{-1}$, the desired estimate holds.
\end{proof}


\noindent It is straightforward to show that \ref{cond:abs6} holds for families of Radon operators. Thus, applying Lemma~\ref{lem:tensor} and Lemma~\ref{lem:bump} several times we deduce that the family $\{T_{\RR^d}[\tilde{\mathfrak m}_{s}^\PPP] : s \in \RR_+^k \}$ satisfies \ref{cond:abs6} as well.

We have just proved \eqref{eq:sigma} for each $p \in (1,2]$. Since the dual conditions specified in Theorem~\ref{thm:2} are naturally satisfied, we conclude that \eqref{eq:sigma} holds for all $p \in (1, \infty)$ as claimed.  

\subsection{Final inductive argument} Since the case $d=1$ is now complete, in order to finish the proof of Theorem~\ref{thm:3} it remains to prove the following result, which enables the inductive step.

\begin{proposition}
	Fix $p \in (1, \infty)$, $d, k \in \ZZ_+$, and let $\{\calM^\PPP_s : s\in\RR_+^k\}$ be the family of Radon operators associated with some fixed $d \times k$ matrix $\mathbb M$ which has no zero columns and rows. Assume that for each $(d_0, k_0) \in \ZZ_+^2$ such that $k_0 \leq k$ and $d_0 < d$ the following holds. Given $\mathbb M_0$, an arbitrary $d_0 \times k_0$ matrix with entries from $\NN$ which has no zero columns and rows, the associated family $\{\calM^{\PPP_0}_s : s\in\RR_+^{k_0}\}$ satisfies
	\begin{equation}\label{eq:tildeP}
	\sup_{I \in \mathfrak S(\RR^k_+)} \big\| O^2_I\big(\calM^{\PPP_0}_s f : s\in\RR^k_+\big) \big\|_{L^p(\RR^{d_0})}\lesssim_{p,\PPP_0} \| f\|_{L^p(\RR^{d_0})}, \qquad f\in L^p(\RR^{d_0}).
	\end{equation}
	 Then the following oscillation inequality holds
	\[
	\sup_{I \in \mathfrak S(\RR^k_+)} \big\| O^2_I\big(\calM^\PPP_s f : s\in\RR^k_+\big) \big\|_{L^p(\RR^{d}) }\lesssim_{p,\PPP} \| f\|_{L^p(\RR^{d})}, \qquad f\in L^p(\RR^d).
	\]
\end{proposition}

\begin{proof}	
	In view of \eqref{eq:mudecomp} and \eqref{eq:sigma} it suffices to prove that for any nonempty $D \subseteq [d]$ we have
	\[
	\sup_{I \in \mathfrak S(\RR^k_+)} \big\| O^2_I\big(T_{\RR^d}[\mathfrak m^{\PPP,D}_s] f : s\in\RR_+^k\big)\big\|_{L^p(\RR^{d})}\lesssim_{p,\PPP} \|f\|_{L^p(\RR^{d})}.
	\]
	Since for $D = [d]$ the claim holds by \eqref{thm:9}, we can assume $d \geq 2$ and $D=[d]\setminus[d_0]$ for some $d_0 \in [d-1]$. Arguing as for bumps, for fixed $s^\circ \in \RR_+^k$ and an arbitrary set $\mathbb{S} \subseteq \RR_+^k$, we have
	\begin{align*}
	& \sup_{s \in \mathbb{S}} \, 
	\big\vert 
	T_{\RR^d} \big[ \mathfrak m_s^{\PPP, D}  - \mathfrak m_{s^\circ}^{\PPP, D} \big] f \big\vert
        = \sup_{s \in \mathbb{S}} \, \big\vert 
	T_{\RR^d} \big[ \Upsilon^{\PPP,D}_{s} \mathfrak m_{s}^{\PPP}\circ \pi^{D} - \Upsilon^{\PPP,D}_{s^\circ} \mathfrak m_{s^\circ}^{\PPP}\circ \pi^{D} \big] f \big\vert \\
	& \quad = \sup_{s \in \mathbb{S}} \, \big\vert 
	T_{\RR^d} \big[ \Upsilon^{\PPP,D}_{s} \big((\mathfrak m_{s}^{\PPP} - \mathfrak m_{s^\circ}^{\PPP})\circ \pi^{D} \big) \big] f \big\vert 
	+ \sum_{i \in D} \, \sup_{s \in \mathbb{S}} \, \big\vert 
	T_{\RR^d} \big[ \Upsilon^{\PPP,[i-1]}_{s} \big( \Upsilon^{\PPP,\{i\}}_{s} - \Upsilon^{\PPP,\{i\}}_{s^\circ} \big)
	\Upsilon^{\PPP,D\setminus[i]}_{s^\circ} \mathfrak m_{s^\circ}^{\PPP}\circ\pi^D \big] f \big\vert \\
	& \quad \lesssim \calM_{\ast}^{D} \Big( \sup_{s \in \mathbb{S}} \, \big\vert T_{\RR^d} \big[ (\mathfrak m_{s}^{\PPP} - \mathfrak m_{s^\circ}^{\PPP})\circ \pi^{D}\big] f \big\vert \Big) 
	+ \sum_{i \in D} \, T_{\RR^d} \big[\mathfrak m_{s^\circ}^{\PPP}\circ \pi^{D} \big] \calM_{\ast}^{D\setminus\{i\}} \Big( \sup_{s \in \mathbb{S}} \, \big\vert T_{\RR^d} \big[\Upsilon^{\PPP,\{i\}}_{s} - \Upsilon^{\PPP,\{i\}}_{s^\circ} \big] f \big\vert \Big),
	\end{align*}
	where $\calM_*^D$ is the composition of $\calM_*^{(i)}$, for $i \in D$, with indices  arranged in the decreasing order. Invoking the Fefferman--Stein inequality for $\calM_{\ast}^{(i)}$ and Lemma~\ref{lem:duo+rubiom} for the family $\{ T_{\RR^d} [\mathfrak m_{I_j}^{\PPP}\circ\pi^{D} ] : j \in \ZZ_+\}$, we obtain
	\begin{align}
        \label{eq:9}
        \begin{split}
        &\big\|O^2_I\big(T_{\RR^d}[\mathfrak m^{\PPP,D}_s] f : s\in\RR_+^k\big)\big\|_{L^p(\RR^{d})} \\
	&\qquad \lesssim_{p, \PPP}
	\big\|O^2_I\big(T_{\RR^d}[\mathfrak m^{\PPP}_s\circ \pi^D] f : s\in\RR_+^k\big)\big\|_{L^p(\RR^{d})}
	+ \sum_{i \in D} 
	\big\|O^2_I\big(T_{\RR^d}[\Upsilon^{\PPP,\{i\}}_s] f : s\in\RR_+^k\big)\big\|_{L^p(\RR^{d})}.
        \end{split}
	\end{align}
	Note that Lemma~\ref{lem:2m} and the inductive hypothesis \eqref{eq:tildeP} were used in the application of Lemma~\ref{lem:duo+rubiom}. Using Lemma~\ref{lem:2m} and \eqref{eq:tildeP} again we see that
	\[
	\big\|O^2_I\big(T_{\RR^d}[\mathfrak m^{\PPP}_s\circ\pi^{D}] f : s\in\RR_+^k\big)\big\|_{L^p(\RR^{d})} 
	\lesssim_{p,\PPP}  \norm{f}_{L^p(\RR^{d})}
	\]
	(it is important that $k_0 < k$ is allowed in \eqref{eq:tildeP}, since zero columns may appear after removing some rows from $\mathbb M$).
	Finally, the relevant bound for the second term in \eqref{eq:9} follows from Lemma~\ref{lem:2m} and \eqref{thm:9}.	
\end{proof}

\end{document}